\documentclass[twosides, ammssymb, 10pt]{amsart}
\usepackage{amssymb}
\usepackage{txfonts}
\usepackage{mathrsfs}
\usepackage{amsmath}
\usepackage{amsfonts}
\usepackage{amsfonts,amssymb}
\usepackage{amscd}
\usepackage{epsf}
\usepackage{amsthm,amsmath}
\usepackage[dvips]{graphicx}
\usepackage{xypic}
\usepackage{latexsym,amssymb}  
\usepackage{graphicx}

\def\A{{\mathscr{A}}}

\def\otm{\otimes}
\def\C{\mathbb C}
\def\R{\mathbb R}
\def\N{\mathbb N}
\def\Z{\mathbb Z}

\def\O{\mathcal{O}}
\def\gl{\mathfrak{gl}_n}

\def\sp{\rm sp}

\def\End{\rm End}
\def\H{\mathcal H}
\def\E{\mathcal E}
\def\F{\mathcal F}

\def\ch{\mathfrak h^*}
\def\Og#1{\mathcal O(#1)}
\def\Ogc#1#2{\mathcal O_{#1}(#2)}
\def\Lo{L_n}
\def\Ld{L_n^{\ast}}
\def\mf{\mathfrak}
\def\sgn{{\rm sgn}}
\def\Id{{\rm Id}}

\newcommand{\oplusop}[1]{{\mathop{\oplus}\limits_{#1}}}

\newtheorem{thm}{Theorem}[section]

\newtheorem{prop}[thm]{Proposition}

\newtheorem{defn}[thm]{Definition}

\numberwithin{equation}{section}

\date{}
\begin{document}

\thispagestyle{empty}

\begin{center}
{\bf{\LARGE  A Categorification of the Spin Representation of
$U(\mf{so}(7,\C))$ via Projective Functors} \footnotetext { $\dag$
Supported by National Natural Science Foundation of Beijing (Grant.
1122006)

\small Email: yjxu2002@163.com,\quad slyang@bjut.edu.cn }}

\bigbreak

\normalsize Yongjun Xu,  Shilin Yang$^{\dag}$

{\footnotesize\small\sl College of Applied Sciences,
Beijing University of Technology\\
\footnotesize\sl Beijing 100124, P. R.  China }
\end{center}

\begin{quote}
{\noindent{\bf Abstract.} The purpose of this paper is to study a
categorification of the $n$-th tensor power of the spin
representation of $U(\mf{so}(7,\C))$ by using certain subcategories
and projective functors of the BGG category of the complex Lie
algebra $\mf{gl}_n$.}
\end{quote}



\section{Introduction}
The general idea of categorification was introduced by Crane and
Frenkel \cite{C,CF}. In recent years categorifications of
algebras and their representations have been studied by many
mathematicians, see for example \cite{KMS479,M,MM,R} and references
therein.

 Let $\Og{\mf{g}}$ be the BGG category associated to a
triangular decomposition of a finite dimensional complex semisimple
or reductive Lie algebra $\mf{g}$.
  The BGG category $\Og{\mf{g}}$ and its projective functors
 play an important role in many algebraic categorifications,
 which can be seen from the following two facets.
On one hand, the projective functors of $\Og{\mf{g}}$ are
extensively used in catetegorifications
 of group algebras of finite Weyl groups and their Hecke algebras.
 In \cite{KMS479}, Khovanov, $et~al.$
 presented several examples about
categorifications of various representations of the symmetric group
$S_n$ via projective functors acting on certain subcategories of the
BGG category $\Og{\mf{sl}_n}$. Especially in \cite{KMS1163} they
categorified integral Specht modules over $S_n$ and
its Hecke algebra by some translation functors of
$\Og{\mf{sl}_n}$. Mazorchuk and Stroppel \cite{MS1363}
constructed a subcategory of $\Og{\mf{g}}$ on which the actions of
translation functors categorify (right) cell modules and induced
cell modules for Hecke algebras of finite Weyl groups. Basing on the
results in \cite{MS1363}, they \cite{MS1}
 gave a categorification of Wedderburn¡¯s
basis for $\C[S_n]$. Moreover, Mazorchuk and Miemietz \cite{MM}
 reproved and extended the result in
\cite{MS1363} by studying 2-representations of abstract 2-categories
from a more systematic and more abstract prospective. In addition,
Mazorchuk and Stroppel \cite{MS2939} applied graded versions of
translation functors and a subcategory of the principal block of
$\Og{\mf{g}}$ to categorifications of a parabolic Hecke module (see
also \cite{S}). On the other hand, the BGG category $\Og{\mf{g}}$
and its projective functors can be applied to categorificaitons of
universal enveloping algebras of simple Lie algebras. In \cite{BFK}
Bernstein, $et~al.$ investigated a categorification of the $n$-th
tensor power of the fundamental representation of $U(\mf{sl}_2)$ via
certain singular blocks and projective functors of
$\Og{\mf{gl}_n}$ (see also
\cite{KMS479}). Following \cite{BFK} Sussan \cite{SU} generalized
the case of $\mf{sl}_2$ in \cite{BFK} to that of $\mf{sl}_k$ and
studied $\mf{sl}_k$-link invariants.

In \cite{BFK} Bernstein, $et~al.$ raised a more difficult problem:
categorifications of the representation theory of arbitrary simple
Lie algebra $\mf{g}$. The main purpose of this article is to study a
categorification of the $n$-th tensor power of the spin
representation of $U(\mf{so}(7,\C))$. The main tools for our
categorification are also the BGG category $\Og{\mf{gl}_n}$ and its
projective functors. Our work can be considered as a part of
categorifications of the representation theory of $U(\mf{g})$ for
the simple Lie algebra $\mf{g}$ of type $B_3$. In other words, we
categorify the image of $U(\mf{so}(7,\C))$ under the algebra
homomorphism $\Phi:U(\mf{so}(7,\C))\to {\End}(V_{\sp}^{\otimes n})$
corresponding to the $n$-th tensor power of the spin representation
$V_{\sp}$ of $U(\mf{so}(7,\C))$. In fact, as standard
representations of the special orthogonal Lie algebras
$\mf{so}(m,\C)$, the spin representations are especially important
since they not only play a fundamental role in the realization of
exceptional simple Lie algebras but also have many important
applications in Lie group, geometry and physics (see \cite{LM}
Chapters I.5 and IV.9, and \cite{FH} Chapter 20).

This paper is organized as follows. In Section \ref{sect-2}, we
collect the background material that will be necessary in the
sequel. In Section \ref{sect-3}, we obtain a categorification of the
$n$-th tensor power of the spin representation of
$U(\mf{so}(7,\C))$. First, we categorify the underlying space of the
$n$-th tensor power $V_{\sp}^{\otm n}$ of the spin representation
$V_{\sp}$ for $U(\mf{so}(7,\C))$ by using certain subcategories of
 $\O(\mf{gl}_n)$ (Theorem \ref{t1}). Next we yield a
categorification of the $U(\mf{so}(7,\C))$ action on $V_{\sp}^{\otm
n}$ by projective functors of $\O(\mf{gl}_n)$ (Theorem \ref{c1}).
Finally, we lift defining relations of $U(\mf{so}(7,\C))$ to
 natural isomorphisms between functors (Theorem \ref{t3}).

Throughout, we denote by $\C,$ $\R$ and $\Z$ the complex number
field, the real number field and the set of integers respectively.

\section{Preliminaries}\label{sect-2}

We start by reviewing some basic results about the universal
enveloping algebra of Lie algebra $\mf{so}(7,\C)$ and the BGG
category of a complex reductive Lie algebra.

As an associative algebra, the universal enveloping algebra
$U(\mf{so}(7,\C))$ of the special orthogonal Lie algebra
$\mf{so}(7,\C)$ is generated by $h_i, e_i, f_i (1\leq i\leq 3)$ over
$\C$ which are subject to the following relations:
\begin{eqnarray*}h_i h_j=h_j h_i,\quad e_i f_j-f_j e_i=\delta_{i,j}h_i,\end{eqnarray*}
\begin{eqnarray*}h_i e_j-e_j h_i=a_{i,j}e_j,\quad h_i f_j-f_j h_i=-a_{i,j}f_j,\end{eqnarray*}
\begin{eqnarray*}{\mathop{\sum}\limits_{k=0}^{1-a_{i,j}}}
(-1)^k \Big( {\begin{array}{*{20}c}
   1-a_{i,j}  \\
   k  \\
\end{array}} \Big)
e_i^{1-a_{i,j}-k}e_j e_i^k=0~for~i\neq j,
\end{eqnarray*}
\begin{eqnarray*}
{\mathop{\sum}\limits_{k=0}^{1-a_{i,j}}}
(-1)^k \Big( {\begin{array}{*{20}c}
   1-a_{i,j}  \\
   k  \\
\end{array}} \Big)
f_i^{1-a_{i,j}-k}f_j f_i^k=0~for~i\neq j,
\end{eqnarray*} where
$a_{i,j}(1\leq i,j\leq 3)$ are the entries of the Cartan matrix
$A=(a_{i,j})_{3\times 3}$ of $\mf{so}(7,\C)$.

Let $V_{\sp}={\mathop{\oplus}\limits_{i=0}^{7}}\C v_i$ be an
8-dimensional vector space over $\C$. Then $V_{\sp}$ is a
$U(\mf{so}(7,\C))$-module in the following way:
\begin{eqnarray*}
&& h_1 v_7=0,~h_1 v_6=0,~h_1 v_5=v_5,~h_1 v_4=v_4,~h_1 v_3=-v_3,
~h_1 v_2=-v_2, ~h_1 v_1=0,~ h_1 v_0=0,\\
&&h_2 v_7=0,~ h_2 v_6=v_6,~ h_2 v_5=-v_5,~ h_2 v_4=0,~ h_2
v_3=0,~h_2
v_2=v_2, ~h_2 v_1=-v_1, ~ h_2 v_0=0,\\
&&h_3 v_7=v_7, ~ h_3 v_6=-v_6, ~ h_3 v_5=v_5, ~ h_3 v_4=-v_4, ~h_3
v_3=v_3, ~ h_3 v_2=-v_2, ~ h_3 v_1=v_1, ~ h_3 v_0=-v_0,\\
&&e_1 v_7=0, ~ e_1 v_6=0, ~ e_1 v_5=0, ~ e_1 v_4=0, ~ e_1 v_3=v_5, ~
e_1 v_2=v_4, ~ e_1 v_1=0, ~ e_1 v_0=0,\\
&&e_2 v_7=0, ~ e_2 v_6=0, ~ e_2 v_5=v_6, ~ e_2 v_4=0, ~
e_2 v_3=0, ~ e_2 v_2=0, ~ e_2 v_1=v_2, ~ e_2 v_0=0,
\end{eqnarray*}
\begin{eqnarray*}
&&e_3 v_7=0, ~ e_3 v_6=v_7, ~ e_3 v_5=0, ~ e_3 v_4=v_5, ~ e_3 v_3=0,
~ e_3 v_2=v_3, ~ e_3 v_1=0, ~ e_3
v_0=v_1,\\
&&f_1 v_7=0, ~ f_1 v_6=0, ~ f_1 v_5=v_3, ~ f_1
v_4=v_2, ~ f_1 v_3=0, ~ f_1 v_2=0, ~ f_1 v_1=0, ~ f_1 v_0=0,\\
&&f_2 v_7=0, ~ f_2 v_6=v_5, ~ f_2 v_5=0, ~ f_2 v_4=0, ~
f_2 v_3=0, ~ f_2 v_2=v_1, ~ f_2 v_1=0, ~ f_2 v_0=0,\\
&&f_3 v_7=v_6, ~ f_3 v_6=0, ~ f_3 v_5=v_4, ~ f_3 v_4=0, ~ f_3
v_3=v_2, ~ f_3 v_2=0, ~ f_3 v_1=v_0, ~ f_3 v_0=0.
\end{eqnarray*}
The $U(\mf{so}(7,\C))$-module $V_{\sp}$ is called the spin
representation of $U(\mf{so}(7,\C))$.

For convenience, we fix some notations which we need in the sequel.
All Lie algebras and their representations are defined over $\C$.
Let $\mf{g}$ be a finite dimensional reductive Lie algebra with a
fixed triangular decomposition $\mf{g}=\mf{n}_+\oplus \mf{h}\oplus
\mf{n}_-.$ Denote by $U(\mf{g})$ the universal enveloping algebra of
$\mf{g}$, $Z(U(\mf{g}))$ the center of $U(\mf{g})$ and $\Theta$ the
set of the central characters. $W$ denotes the Weyl group of $\mf
g$. $\rho$ is the half-sum of positive roots. Define the dot-action
of $W$ on $\ch$ as follows: $w\cdot \lambda=w(\lambda+\rho)-\rho$.
For $\lambda\in \ch$, let
$\theta_{\lambda}:Z(U(\mf{g}))\rightarrow\C$ be the corresponding
central character and $M(\lambda)$ the Verma module with the highest
weight $\lambda.$

Let $\ch_{dom}$ be the set of all elements in $\ch$ dominant with
respect to the dot-action. Then there is a map
$\eta:\ch\rightarrow\Theta$ which maps $\lambda$ to
$\theta_{\lambda}$ sets up a bijection between $\ch_{dom}$ and
$\Theta$ (see \cite{Di}, Section 7.4). The notation $\Og{\mf{g}}$
denotes the BGG category of $\mf{g}$ associated to the
 triangular decomposition $\mf{g}=\mf{n}_+\oplus \mf{h}\oplus\mf{n}_-.$
For any $\theta\in \Theta$ denote by $\Ogc{\theta}{\mf{g}}$ the full
subcategory of $\Og{\mf{g}}$ whose objects are the modules $M$ where
$$M=\bigg\{m\in M\big|(z-\theta(z))^{n}\cdot m=0~{\rm for~some~}n\in
\N~{\rm for~each}~z\in Z(U(\mf{g}))\bigg\}.$$ The BGG category
$\Og{\mf{g}}$ is the direct sum of the subcategories
$\Ogc{\theta}{\mf{g}}$ as $\theta$ ranges over the central
characters of the form $\theta_{\lambda}$ (see \cite{H}, Section
1.12).

Now we give a brief introduction to projective functors.

Denote by $\textrm{proj}_{\theta}$ the functor from $\Og{\mf{g}}$ to
$\Og{\mf{g}}$ that, to a module $M=\oplusop{\theta\in
\Theta}M(\theta),$ associates the $\theta$-component summand
$M(\theta)$ of $M$. Let $F_V$ be the functor of tensoring with a
finite-dimensional $\mf{g}$-module $V.$

\begin{defn}
$F:\Og{\mf{g}} \to \Og{\mf{g}}$ is a projective functor if it is
isomorphic to a direct summand of the functor $F_V$ for some finite
dimensional $\mf{g}$-module $V.$
\end{defn}

Denote by $K(\A)$ the Grothendieck group of an abelian or
triangulated category $\A.$ It is the free abelian group generated
by the symbols $[M]$ where $M$ is an object of $\A.$ The only
relations in this group are of the form $[N]=[M]+[P]$ when there is
a short exact sequence or distinguished triangle of the form $0 \to
M \to N \to P \to 0.$ The image of an object $M$ and an exact
functor $F$ in the Grothendieck group will be denoted by $[M]$ and
$[F]$ respectively.

The following properties of projective functors can be found in
Section 3.2 or 3.4 of \cite{BG}.

\begin{prop}
\label{p3}
\begin{enumerate}
\item
Projective functors are exact.
\item
Any direct sum of projective functors is a projective functor.
\item
Any composition of projective functors is a projective functor.
\item
The functor $\textrm{\rm proj}_{\theta}:\Og{\mf{g}} \to \Og{\mf{g}}$
is a projective functor.
\item
Let $F,G$ be projective functors. If $[F]=[G],$ then $F\cong G.$
\end{enumerate}
\end{prop}

Fix a central character $\theta,$ then
\begin{eqnarray}\label{base}
\bigg\{~[M(\lambda)]~\big|~\theta=\theta_{\lambda}~\bigg\}
=\bigg\{~[M(\mu)]~\big|~\mu\in W\cdot\lambda~\bigg\}
\end{eqnarray}
forms a $\Z$-basis of the Grothendieck group
$K(\Ogc{\theta}{\mf{g}})$ (see \cite{H}, Section 1.10 and 1.12). The
following proposition shows that this basis is handy for writing the
action of projective functors on the Grothendieck group of
$\O(\mf{g})$(see \cite{BFK}, Section 2.3.2, and \cite{BG}, Section
1.12).

\begin{prop}
\label{p4} Let $V$ be a finite-dimensional $\mf{g}$-module,
$\mu_1,\cdots,\mu_m$ the multiset of weights of $V$, i.e., there is
a basis $v_1,v_2,\cdots, v_m$ of $V$ such that the weight of the
vector $v_i$ equals $\mu_i$,
 $M(\lambda)$ the Verma
module with the highest weight $\lambda,$ then we have $[V\otimes
M(\lambda)]=\sum_{i=1}^m [M(\lambda+\mu_i)]$ in the Grothendieck
group $K(\Og{\mf{g}}).$
\end{prop}

For unexplained concepts and notations, we refer the reader to
\cite{BG,BGG,FH,HK,H,M}.

\section{Categorification of the Spin Representation of
$U(\mf{so}(7,\C))$}\label{sect-3}

This section is to obtain a categorification of the
$n$-th tensor power $V_{\sp}^{\otm n}$ of the spin representation
$V_{\sp}$ for $U(\mf{so}(7,\C))$ in the following
three steps.
\begin{itemize}
\item[(1)]  Categorifying the underlying space
of the $n$-th tensor power $V_{\sp}^{\otm n}$ of the spin
representation $V_{\sp}$ of $U(\mf{so}(7,\C))$ by using certain
subcategories of the BGG category $\O(\mf{gl}_n)$.
\item[(2)] Yielding a
categorification of the $U(\mf{so}(7,\C))$ action on $V_{\sp}^{\otm
n}$ by projective functors of $\O(\mf{gl}_n)$.
\item[(3)] Lifting defining relations of $U(\mf{so}(7,\C))$
 to natural isomorphisms between functors.
\end{itemize}

We fix once and for all a triangular decomposition
$\mathfrak{n}_+\oplus \mathfrak{h}\oplus \mathfrak{n}_-$ of the Lie
algebra $\gl.$ The Weyl group of $\mf{gl_n}$ is isomorphic to the
symmetric group $S_n.$ Choose a standard orthogonal basis
$\varepsilon_1,\cdots,\varepsilon_n$ in the Euclidean space $\R^n$
and identify the complexification $\C\otm_{\R}\R^n$ with the dual
$\ch$ of Cartan subalgebra so that $R_+=\{~
\varepsilon_i-\varepsilon_j~|~1\leq i < j\leq n~\}$ is  the set of
positive roots and $\beta_i=\varepsilon_i-\varepsilon_{i+1},1\le
i\le n-1$ are simple roots. The generator $s_i$ of the Weyl group
$W=S_n$ acts on $\ch$ by permuting
 $\varepsilon_i$ and $\varepsilon_{i+1}.$
Denote by $\rho$ the half-sum of positive roots
$$\rho=\frac{n-1}{2}\varepsilon_1+\frac{n-3}{2}\varepsilon_2+\cdots+\frac{1-n}{2}\varepsilon_n.$$
We denote by $[0, 7]$ the integers $0\leq k\leq 7.$ For a sequence
$(a_1,\cdots, a_n)\in[0, 7]^n$ we denote by $M(a_1,\cdots, a_n)$ the
Verma module with the highest weight
$a_1\varepsilon_1+\cdots+a_n\varepsilon_n-\rho.$

Let ${\bf D}$ be the set of all 8-tuples of nonnegative integers
${\bf d}=(d_0,d_1,d_2,d_3,d_4,d_5,d_6,d_7)$ such that
${\mathop{\sum}\limits_{k=0}^{7}}d_k=n$. We define the following
equivalence relation $\thicksim$ on ${\bf D}$:
\begin{eqnarray*}
{\bf d}\thicksim {\bf d^{'}}&\Leftrightarrow&\left\{
 \begin{array}{ll}
  d_7+d_6+d_5+d_4-d_3-d_2-d_1-d_0=d_7^{'}+d_6^{'}+d_5^{'}+d_4^{'}-d_3^{'}-d_2^{'}-d_1^{'}-d_0^{'}, &\\
 d_7+d_6-d_5-d_4+d_3+d_2-d_1-d_0=d_7^{'}+d_6^{'}-d_5^{'}-d_4^{'}+d_3^{'}+d_2^{'}-d_1^{'}-d_0^{'},& \\
 d_7-d_6+d_5-d_4+d_3-d_2+d_1-d_0=d_7^{'}-d_6^{'}+d_5^{'}-d_4^{'}+d_3^{'}-d_2^{'}+d_1^{'}-d_0^{'},&
 \end{array}
\right.
\end{eqnarray*}
for any ${\bf d}=(d_0,d_1,d_2,d_3,d_4,d_5,d_6,d_7),{\bf
d^{'}}=(d_0^{'},d_1^{'},d_2^{'},d_3^{'},d_4^{'},d_5^{'},d_6^{'},d_7^{'})\in
{\bf D}$. In the following $[{\bf d}]$ and $\widetilde{{\bf D}}$
represent the equivalent class of ${\bf d}$ and the set of all the
the equivalent classes respectively.

The spin representation $V_{\sp}$ has the weight space decomposition
$V_{\sp}={\mathop{\oplus}\limits_{k=0}^{7}}V_{k},$ where $V_{k}=\C
v_k$~for~$0\leq k\leq 7$ (see \cite{HK} Chapter 2). For ${\bf d'}\in
{\bf D}$ and $(a_1,\cdots, a_n)\in[0, 7]^n$, we define the following
condition:
\begin{eqnarray}\label{eqn3-1}\sharp~ \{~a_m~|~a_m=k, 1\leq m \leq n\}
=d_k^{'}~for~0\leq k\leq 7.
\end{eqnarray}
It follows that $V_{\sp}^{\otm n}$ has the weight space
decomposition $V_{\sp}^{\otm n}={\mathop{\oplus}\limits_{[{\bf
d}]\in \widetilde{{\bf D}}}}(V_{\sp}^{\otm n})_{[{\bf d}]},$ where
$(V_{\sp}^{\otm n})_{[{\bf d}]}$ is the $\C$-linear space spanned by
$$B^{'}_{[{\bf d}]}:=\bigg\{ v_{a_{1}}\otimes v_{a_{2}}\otimes\cdots
\otimes v_{a_{n}} \big{|} (a_1, \cdots, a_n)\in [0, 7]^n \hbox{
satisfying the conditon (\ref{eqn3-1}) for some}~{\bf d'}\in[{\bf
d}]\bigg\}.$$

From now on, we denote by $^{\Z}(V_{\sp}^{\otm n})_{[{\bf d}]}$ the
$\Z$-module spanned by $B^{'}_{[{\bf d}]}$ and $^{\Z}V_{\sp}^{\otm
n}:={\mathop{\oplus}\limits_{[{\bf d}]\in \widetilde{{\bf
D}}}}{^{\Z}(V_{\sp}^{\otm n})_{[{\bf d}]}}.$ It is easy to see that
$\C\otm_{\Z}{^{\Z}(V_{\sp}^{\otm n})_{[{\bf d}]}}=(V_{\sp}^{\otm
n})_{[{\bf d}]}$ and $\C\otm_{\Z}{^{\Z}V_{\sp}^{\otm
n}}=V_{\sp}^{\otm n}.$ For each ${\bf d}\in {\bf D}$, set
$\lambda_{\bf d}={\mathop{\sum}\limits_{i=0}^{7}}
{\mathop{\sum}\limits_{j=1}^{d_i}} (7-i) \varepsilon_{d_0 + \cdots +
d_{i-1} + j}.$ Denote by $\theta_{\bf d}=\eta(\lambda_{\bf d}-\rho)$
the corresponding central character of $\mf{gl}_n$ under the map
$\eta: \ch\rightarrow\Theta$. We define $\O_{\bf d}:=\O_{\theta_{\bf
d}}(\mf{gl}_n),$ $\O_{[{\bf d}]}:={\mathop{\oplus}\limits_{{\bf
d^{'}}\in[{\bf d}]}}\O_{\bf d^{'}}$ and
$\O^{n}:={\mathop{\oplus}\limits_{[{\bf d}]\in \widetilde{{\bf
D}}}}\O_{[{\bf d}]}.$

Now we are prepared to
realize $^{\Z}V_{\sp}^{\otm n}$ and its weight space
$^{\Z}(V_{\sp}^{\otm n})_{[{\bf d}]}$ for any $[{\bf d}]\in
\widetilde{{\bf D}}$ as the Grothendieck groups of the categories
$\O^n$ and $\O_{[{\bf d}]}$ respectively. Indeed, we have the
following result.

\begin{thm}
\label{t1} There exists an isomorphism of abelian groups
$\gamma_n:K(\O^n)\rightarrow ^{\Z}V_{\sp}^{\otm n}$ given by
\[\gamma_n([M(a_1,\cdots, a_n)])=v_{a_{1}}\otimes v_{a_{2}}\otimes\cdots \otimes v_{a_{n}}\]
for any sequence $(a_1,\cdots, a_n)\in[0, 7]^n.$  Moreover, the
restriction of $\gamma_n$ on $K(\O_{[{\bf d}]})$ is an abelian group
 isomorphism between $K(\O_{[{\bf d}]})$ and
$^{\Z}(V_{\sp}^{\otm n})_{[{\bf d}]}$ for any $[{\bf d}]\in
\widetilde{{\bf D}}$.
\end{thm}

\begin{proof}  To prove the theorem, it is sufficient to prove that
$\gamma_n:K(\O_{[{\bf d}]})\otimes_{\Z} \C\rightarrow
^{\Z}(V_{\sp}^{\otm n})_{[{\bf d}]}$ is an abelian group isomorphism
for any $[{\bf d}]\in \widetilde{{\bf D}}$. Indeed, the above
abelian group isomorphism will be obvious if we note the following
facts.

For any $[{\bf d}]\in \widetilde{{\bf D}}$ and ${\bf d^{'}}\in[{\bf
d}]$ it is seen from (\ref{base}) that the set of all the symbols
$[M(a_1,\cdots, a_n)]$ satisfying the condition (\ref{eqn3-1}) is a
$\Z$-basis of the Grothendieck group $K(\O_{\bf d^{'}}).$ We denote
this $\Z$-basis by $B_{\bf d^{'}}$. It follows that $B_{[{\bf
d}]}={\mathop{\cup}\limits_{{\bf d^{'}}\in {\bf [d]}}}B_{\bf d^{'}}$
is a $\Z$-basis of the Grothendieck group $K(\O_{[{\bf d}]})$. On
the other hand, if we denote by $B'_{\bf d^{'}}$ the set of
$v_{a_{1}}\otimes v_{a_{2}}\otimes\cdots \otimes v_{a_{n}}$ such
that the sequence $(a_1, \cdots, a_n)\in [0, 7]^n$ satisfies
(\ref{eqn3-1}), then $B^{'}_{[{\bf d}]}:={\mathop{\cup}\limits_{{\bf
d^{'}}\in [{\bf d}]}}B'_{\bf d^{'}}$ is a $\Z$-basis of the weight
space $^{\Z}(V_{\sp}^{\otm n})_{[{\bf d}]}$.
\end{proof}

Let $\Lo$ be the $n$-dimensional fundamental representation of $\gl$
with weights $\varepsilon_1,\varepsilon_2,\cdots ,\varepsilon_n$ and
the corresponding weight vectors $u_1, u_2, \cdots, u_n$. Then its
dual representation $\Ld$ has weights
$-\varepsilon_1,-\varepsilon_2,\cdots, -\varepsilon_n.$ In addition,
we recall some submodules of $\Lo^{\otimes 2}$ and $(\Ld)^{\otimes
2}$ which will be used to construct the functors for our
categorification. Denote by $Sym^2(\Lo)$ the symmetric square of
$\Lo$, i.e, the submodule of $\Lo^{\otimes 2}$ spanned by
$u_i\otimes u_i(1\leq i\leq n)$ and $u_i\otimes u_j+u_j\otimes
u_i(1\leq i< j\leq n)$, and denote by $Alt^2(\Lo)$ the alternative
square of $\Lo$, i.e, the submodule of $\Lo^{\otimes 2}$ spanned by
$u_i\otimes u_j-u_j\otimes u_i(1\leq i< j\leq n)$. Similarly, we
denote by $Sym^2(\Ld)$ and $Alt^2(\Ld)$ the symmetric square and the
alternative square of $\Ld$ respectively. In the following, we
define $\O_{\bf d}$ to be the trivial subcategory of $\O(\gl)$ for
${\bf d}\notin {\bf D}.$ For ${\bf d}\in {\bf D}$ let ${\bf d}_i$
denote the fact that one subtracts 1 from the coefficient at place
$i,$ and ${\bf d}^i$ the fact that one adds 1 to the coefficient at
place $i.$ Then ${\bf d}_i^j$ means that one subtracts 1 from the
coefficient at place $i$ and adds 1 to the coefficient at place $j.$

To categorify the action of $U(\mf{so}(7,\C))$ on $V_{\sp}^{\otimes
n}$, we introduce a series of projective functors of $\O(\gl).$

For ${\bf d}=(d_0,d_1,d_2,d_3,d_4,d_5,d_6,d_7)\in {\bf D}$, define
\begin{eqnarray*}
&&c_1({\bf d}):=d_5+d_4-d_3-d_2,\\
&&c_2({\bf d}):=d_6-d_5+d_2-d_1,\\
&&c_3({\bf d}):=d_7-d_6+d_5-d_4+d_3-d_2+d_1-d_0, \end{eqnarray*} and
for $1\leq i\leq 3$, denote by $\sgn(c_i({\bf d}))$ the sign
function of $c_i({\bf d})$, i.e.,
\[
\sgn(c_i({\bf d}))=\left\{
 \begin{array}{ll}
  1,~{\rm if}~c_i({\bf d})>0,&\\
 0,~{\rm if}~c_i({\bf d})=0,& \\
 -1,~{\rm if}~c_i({\bf d})<0.&
 \end{array}
\right.
\]
Then set
\[\H_i^{\sgn(c_i({\bf d}))}([{\bf d}])=(\Id_{\O_{[{\bf d}]}})^{\oplus \sgn(c_i({\bf d}))c_i({\bf d})}:
 \O_{[{\bf d}]}\to\O_{[{\bf d}]},\]
where $\Id_{\O_{[{\bf d}]}}$ is the identity functor of $\O_{[{\bf
d}]}$. From the definition of the equivalence relation $\thicksim$
on ${\bf D}$ we can see that the functors $\H_i^{\sgn(c_i({\bf
d}))}([{\bf d}])(1\leq i\leq 3)$ are independent on the choice of
the representative ${\bf d}$ of [{\bf d}].

For ${\bf d}\in {\bf D}$ denote
\[\E_1^{+2}({\bf d}):=\textrm{proj}_{\theta_{{\bf d}_2^4}}
\circ F_{Sym^2(\Lo)}: \O_{\bf d}\to\O_{{\bf d}_2^4},\]
\[\E_1^{-2}({\bf d}):=\textrm{proj}_{\theta_{{\bf d}_2^4}}
\circ F_{Alt^2(\Lo)}: \O_{\bf d}\to \O_{{\bf d}_2^4},\]
\[\E_1^{+3}({\bf
d}):=\textrm{proj}_{\theta_{{\bf d}_3^5}} \circ F_{Sym^2(\Lo)}:
\O_{\bf d}\to \O_{{\bf d}_3^5},\]
\[\E_1^{-3}({\bf d}):=\textrm{proj}_{\theta_{{\bf d}_3^5}}
\circ F_{Alt^2(\Lo)}: \O_{\bf d}\to \O_{{\bf d}_3^5}.\] For any
${\bf [d]}\in \widetilde{{\bf D}}$, we set
\[
\E_{1}^+([{\bf d}])={\mathop{\oplus}\limits_{{\bf d^{'}}\in[{\bf
d}]}}(\E_1^{+2}({\bf d^{'}})\oplus\E_1^{+3}({\bf d^{'}})):
\O_{[{\bf d}]}\to \O_{[\overleftarrow{{\bf d}_1}],}
\]
\[
\E_{1}^-([{\bf d}])={\mathop{\oplus}\limits_{{\bf d^{'}}\in[{\bf
d}]}}(\E_1^{-2}({\bf d^{'}})\oplus\E_1^{-3}({\bf d^{'}})):
\O_{[{\bf d}]}\to \O_{[{\bf \overleftarrow{\bf d_1}}]},
\]
where $[{\bf \overleftarrow{\bf d_1}}]=[{{\bf d}_3^5}] =[{{\bf
d}_2^4}].$

For ${\bf d}\in {\bf D}$ denote
\[\E_2^1({\bf
d}):=\textrm{proj}_{\theta_{{\bf d}_1^2}} \circ F_{\Lo}: \O_{\bf
d}\to \O_{{\bf d}_1^2},\]
 \[\E_2^5({\bf
d}):=\textrm{proj}_{\theta_{{{\bf d}_5^6}}} \circ F_{\Lo}: \O_{\bf
d}\to \O_{{\bf d}_5^6}.\] For any $[{\bf d}]\in \widetilde{{\bf
D}}$, set
\[
\E_{2}([{\bf d}])={\mathop{\oplus}\limits_{{\bf d^{'}}\in[{\bf
d}]}}(\E_2^1({\bf d^{'}})\oplus\E_2^5({\bf d^{'}})):   \O_{[{\bf
d}]}\to \O_{[\bf \overleftarrow{\bf d_2}]},
\]
where ${\bf [\overleftarrow{\bf d_2}]}=[{{\bf d}_1^2}] =[{{\bf
d}_5^6}].$

For ${\bf d}\in {\bf D}$ denote \[\E_3^0({\bf
d}):=\textrm{proj}_{\theta_{{{\bf d}_0^1}}} \circ F_{\Lo}: \O_{\bf
d}\to \O_{{\bf d}_0^1},\]
\[\E_3^2({\bf d}):=\textrm{proj}_{\theta_{{{\bf d}_2^3}}}
\circ F_{\Lo}: \O_{\bf d}\to \O_{{\bf d}_2^3},\]
\[\E_3^4({\bf d}):=\textrm{proj}_{\theta_{{{\bf d}_4^5}}} \circ F_{\Lo}:
\O_{\bf d}\to \O_{{\bf d}_4^5},\]
\[\E_3^6({\bf d}):=\textrm{proj}_{\theta_{{{\bf d}_6^7}}}
\circ F_{\Lo}: \O_{\bf d}\to \O_{{\bf d}_6^7}.\] For any $[{\bf
d}]\in \widetilde{{\bf D}}$, set
\[
\E_{3}([{\bf d}])={\mathop{\oplus}\limits_{{\bf d^{'}}\in[{\bf
d}]}}(\E_3^0({\bf d^{'}})\oplus\E_3^2({\bf d^{'}})\oplus\E_3^4({\bf
d^{'}})\oplus\E_3^6({\bf d^{'}})):  \O_{[{\bf d}]}\to \O_{[\bf
\overleftarrow{\bf d_3}]},
\]
where ${\bf [\overleftarrow{\bf d_3}]} =[{{\bf d}_0^1}] =[{{\bf
d}_2^3}]=[{{\bf d}_4^5}]=[{{\bf d}_6^7}].$

For ${\bf d}\in {\bf D}$ denote \[\F_1^{+4}({\bf
d}):=\textrm{proj}_{\theta_{{{\bf d}_4^2}}} \circ F_{Sym^2(\Ld)}:
\O_{\bf d}\to \O_{{\bf d}_4^2},\]
 \[\F_1^{-4}({\bf
d}):=\textrm{proj}_{\theta_{{{\bf d}_4^2}}} \circ F_{Alt^2(\Ld)}:
 \O_{\bf d}\to \O_{{\bf d}_4^2},\]
\[\F_1^{+5}({\bf
d}):=\textrm{proj}_{\theta_{{{\bf d}_5^3}}} \circ F_{Sym^2(\Ld)}:
 \O_{\bf d}\to \O_{{\bf d}_5^3},\]
 \[\F_1^{-5}({\bf
d}):=\textrm{proj}_{\theta_{{{\bf d}_5^3}}}\circ F_{Alt^2(\Ld)}:
 \O_{\bf d}\to \O_{{\bf d}_5^3}\] For any ${\bf [d]}\in
\widetilde{{\bf D}}$, set
\[\F^+_{1}([{\bf d}])={\mathop{\oplus}\limits_{{\bf
d^{'}}\in[{\bf d}]}}(\F_1^{+4}({\bf d^{'}})\oplus\F_1^{+5}({\bf
d^{'}})):  \O_{[{\bf d}]}\to \O_{\bf [\overrightarrow{\bf d_1}]},\]
\[\F^-_{1}([{\bf d}])={\mathop{\oplus}\limits_{{\bf
d^{'}}\in[{\bf d}]}}(\F_1^{-4}({\bf d^{'}})\oplus\F_1^{-5}({\bf
d^{'}})):  \O_{[{\bf d}]}\to \O_{[\bf \overrightarrow{\bf d_1}]},\]
where ${\bf [\overrightarrow{\bf d_1}]}=[{{\bf d}_4^2}] =[{{\bf
d}_5^3}].$

For ${\bf d}\in {\bf D}$ denote \[\F_2^2({\bf
d}):=\textrm{proj}_{\theta_{{{\bf d}_2^1}}} \circ F_{\Ld}: \O_{\bf
d}\to \O_{{\bf d}_2^1},\]
\[\F_2^6({\bf
d}):=\textrm{proj}_{\theta_{{{\bf d}_6^5}}} \circ F_{\Ld}: \O_{\bf
d}\to \O_{{\bf d}_6^5}.\] For any $[{\bf d}]\in \widetilde{{\bf
D}}$, set
\[\F_{2}([{\bf d}])={\mathop{\oplus}\limits_{{\bf
d^{'}}\in[{\bf d}]}}(\F_2^2({\bf d^{'}})\oplus\F_2^6({\bf d^{'}})):
\O_{[{\bf d}]}\to \O_{\bf [\overrightarrow{\bf d_2}]},\] where ${\bf
[\overrightarrow{\bf d_2}]}=[{{\bf d}_2^1}] =[{{\bf d}_6^5}].$

For ${\bf d}\in {\bf D}$ denote \[\F_3^1({\bf
d}):=\textrm{proj}_{\theta_{{{\bf d}_1^0}}} \circ F_{\Ld}: \O_{\bf
d}\to \O_{{\bf d}_1^0},\]
\[\F_3^3({\bf d}):=\textrm{proj}_{\theta_{{{\bf d}_3^2}}}
\circ F_{\Ld}: \O_{\bf d}\to \O_{{\bf d}_3^2},\]
\[\F_3^5({\bf
d}):=\textrm{proj}_{\theta_{{{\bf d}_5^4}}} \circ F_{\Ld}: \O_{\bf
d}\to \O_{{\bf d}_5^4},\]
\[\F_3^7({\bf
d}):=\textrm{proj}_{\theta_{{{\bf d}_7^6}}} \circ F_{\Ld}: \O_{\bf
d}\to \O_{{\bf d}_7^6}.\] For any $[{\bf d}]\in \widetilde{{\bf
D}}$, set
\[\F_{3}([{\bf d}])={\mathop{\oplus}\limits_{{\bf
d^{'}}\in[{\bf d}]}}(\F_3^1({\bf d^{'}})\oplus\F_3^3({\bf
d^{'}})\oplus\F_3^5({\bf d^{'}})\oplus\F_3^7({\bf d^{'}})):
\O_{[{\bf d}]}\to \O_{\bf [\overrightarrow{\bf d_3}]},\] where ${\bf
[\overrightarrow{\bf d_3}]}=[{{\bf d}_1^0}] =[{{\bf d}_3^2}] =[{{\bf
d}_5^4}] =[{{\bf d}_7^6}].$

It can be seen from Proposition \ref{p3} (1), (2), (3) and (4) that
the above functors we introduce are exact and projective functors.
Therefore, they can induce abelian group homomorphisms of the
corresponding Grothendieck groups. By Proposition \ref{p4} and
direct calculations we can obtain the following explicit formulas of
their induced homomorphisms on the basis element $[M(a_1,\cdots,
a_n)]\in B_{\bf d}$, which will be used in checking the
commutativity of the diagrams in Proposition \ref{t2}.
\begin{eqnarray}\label{E1}
&\quad&[\E_1^{+2}({\bf
d})]([M(a_1,\cdots,a_n)])\\
&=&{\mathop{\sum}\limits_{m=1,\atop
a_{m}=2}^n}[M(a_1,\cdots,a_{m-1},a_{m}+2,a_{m+1},\cdots,a_n)]\nonumber\\
&&+{\mathop{\sum}\limits_{1\leq i<j\leq n,\atop
(a_{i},a_{j})=(2,3)~{\rm
or}~(3,2)}}[M(a_1,\cdots,a_{i-1},a_{i}+1,a_{i+1},\cdots,
a_{j-1},a_{j}+1,a_{j+1},\cdots,a_n)].\nonumber
\end{eqnarray}
\begin{eqnarray}\label{E2}
&\quad&[\E_1^{-2}({\bf d})]([M(a_1,\cdots,a_n)])\\
&=&{\mathop{\sum}\limits_{1\leq i<j\leq n,\atop
(a_{i},a_{j})=(2,3)~{\rm
or}~(3,2)}}[M(a_1,\cdots,a_{i-1},a_{i}+1,a_{i+1},\cdots,
a_{j-1},a_{j}+1,a_{j+1},\cdots,a_n)].\nonumber
\end{eqnarray}
\begin{eqnarray}\label{E3}
&\quad&[\E_1^{+3}({\bf d})]([M(a_1,\cdots,
a_n)])\\
&=&{\mathop{\sum}\limits_{m=1,\atop
a_{m}=3}^n}[M(a_1,\cdots,a_{m-1},a_{m}+2,a_{m+1},\cdots,a_n)]\nonumber\\
&&+{\mathop{\sum}\limits_{1\leq i<j\leq n,\atop
(a_{i},a_{j})=(3,4)~{\rm
or}~(4,3)}}[M(a_1,\cdots,a_{i-1},a_{i}+1,a_{i+1},\cdots,
a_{j-1},a_{j}+1,a_{j+1},\cdots,a_n)].\nonumber
\end{eqnarray}
\begin{eqnarray}\label{E4}
&\quad&[\E_1^{-3}({\bf d})]([M(a_1,\cdots,
a_n)])\\
&=&{\mathop{\sum}\limits_{1\leq i<j\leq n,\atop
(a_{i},a_{j})=(3,4)~{\rm
or}~(4,3)}}[M(a_1,\cdots,a_{i-1},a_{i}+1,a_{i+1},\cdots,
a_{j-1},a_{j}+1,a_{j+1},\cdots,a_n)].\nonumber
\end{eqnarray}
\begin{eqnarray}\label{E5}
[\E_2^1({\bf d})]([M(a_1,\cdots,
a_n)])={\mathop{\sum}\limits_{m=1,\atop
a_{m}=1}^n}[M(a_1,\cdots,a_{m-1},a_{m}+1,a_{m+1},\cdots,a_n)].
\end{eqnarray}
\begin{eqnarray}\label{E6}
[\E_2^5({\bf d})]([M(a_1,\cdots,
a_n)])={\mathop{\sum}\limits_{m=1,\atop
a_{m}=5}^n}[M(a_1,\cdots,a_{m-1},a_{m}+1,a_{m+1},\cdots,a_n)].
\end{eqnarray}
\begin{eqnarray}\label{E7}
[\E_3^0({\bf d})]([M(a_1,\cdots,
a_n)])={\mathop{\sum}\limits_{m=1,\atop
a_{m}=0}^n}[M(a_1,\cdots,a_{m-1},a_{m}+1,a_{m+1},\cdots,a_n)].
\end{eqnarray}
\begin{eqnarray}\label{E8}
[\E_3^2({\bf d})]([M(a_1,\cdots,
a_n)])={\mathop{\sum}\limits_{m=1,\atop
a_{m}=2}^n}[M(a_1,\cdots,a_{m-1},a_{m}+1,a_{m+1},\cdots,a_n)].
\end{eqnarray}
\begin{eqnarray}\label{E9}
[\E_3^4({\bf d})]([M(a_1,\cdots,
a_n)])={\mathop{\sum}\limits_{m=1,\atop
a_{m}=4}^n}[M(a_1,\cdots,a_{m-1},a_{m}+1,a_{m+1},\cdots,a_n)].
\end{eqnarray}
\begin{eqnarray}\label{E10}
[\E_3^6({\bf d})]([M(a_1,\cdots,
a_n)])={\mathop{\sum}\limits_{m=1,\atop
a_{m}=6}^n}[M(a_1,\cdots,a_{m-1},a_{m}+1,a_{m+1},\cdots,a_n)].
\end{eqnarray}
\begin{eqnarray}\label{E11}
&\quad&[\F_1^{+4}({\bf d})]([M(a_1,\cdots,
a_n)])\\
&=&{\mathop{\sum}\limits_{m=1,\atop
a_{m}=4}^n}[M(a_1,\cdots,a_{m-1},a_{m}-2,a_{m+1},\cdots,a_n)]\nonumber\\
&&+{\mathop{\sum}\limits_{1\leq i<j\leq n,\atop
(a_{i},a_{j})=(3,4)~{\rm
or}~(4,3)}}[M(a_1,\cdots,a_{i-1},a_{i}-1,a_{i+1},\cdots,
a_{j-1},a_{j}-1,a_{j+1},\cdots,a_n)].\nonumber
\end{eqnarray}
\begin{eqnarray}\label{E12}
&\quad&[\F_1^{-4}({\bf d})]([M(a_1,\cdots,
a_n)])\\
&=&{\mathop{\sum}\limits_{1\leq i<j\leq n,\atop
(a_{i},a_{j})=(3,4)~{\rm
or}~(4,3)}}[M(a_1,\cdots,a_{i-1},a_{i}-1,a_{i+1},\cdots,
a_{j-1},a_{j}-1,a_{j+1},\cdots,a_n)].\nonumber
\end{eqnarray}
\begin{eqnarray}\label{E13}
&\quad&[\F_1^{+5}({\bf d})]([M(a_1,\cdots,
a_n)])\\
&=&{\mathop{\sum}\limits_{m=1,\atop
a_{m}=5}^n}[M(a_1,\cdots,a_{m-1},a_{m}-2,a_{m+1},\cdots,a_n)]\nonumber\\
&&+{\mathop{\sum}\limits_{1\leq i<j\leq n,\atop
(a_{i},a_{j})=(4,5)~{\rm
or}~(4,5)}}[M(a_1,\cdots,a_{i-1},a_{i}-1,a_{i+1},\cdots,
a_{j-1},a_{j}-1,a_{j+1},\cdots,a_n)].\nonumber
\end{eqnarray}
\begin{eqnarray}\label{E14}
&\quad&[\F_1^{-5}({\bf d})]([M(a_1,\cdots,
a_n)])\\
&=&{\mathop{\sum}\limits_{1\leq i<j\leq n,\atop
(a_{i},a_{j})=(4,5)~{\rm
or}~(5,4)}}[M(a_1,\cdots,a_{i-1},a_{i}-1,a_{i+1},\cdots,
a_{j-1},a_{j}-1,a_{j+1},\cdots,a_n)].\nonumber
\end{eqnarray}
\begin{eqnarray}\label{E15}
[\F_2^2({\bf d})]([M(a_1,\cdots,
a_n)])={\mathop{\sum}\limits_{m=1,\atop
a_{m}=2}^n}[M(a_1,\cdots,a_{m-1},a_{m}-1,a_{m+1},\cdots,a_n)].
\end{eqnarray}
\begin{eqnarray}\label{E16}
[\F_2^6({\bf d})]([M(a_1,\cdots,
a_n)])={\mathop{\sum}\limits_{m=1,\atop
a_{m}=6}^n}[M(a_1,\cdots,a_{m-1},a_{m}-1,a_{m+1},\cdots,a_n)].
\end{eqnarray}
\begin{eqnarray}\label{E17}
[\F_3^1({\bf d})]([M(a_1,\cdots,
a_n)])={\mathop{\sum}\limits_{m=1,\atop
a_{m}=1}^n}[M(a_1,\cdots,a_{m-1},a_{m}-1,a_{m+1},\cdots,a_n)].
\end{eqnarray}
\begin{eqnarray}\label{E18}
[\F_3^3({\bf d})]([M(a_1,\cdots,
a_n)])={\mathop{\sum}\limits_{m=1,\atop
a_{m}=3}^n}[M(a_1,\cdots,a_{m-1},a_{m}-1,a_{m+1},\cdots,a_n)].
\end{eqnarray}
\begin{eqnarray}\label{E19}
[\F_3^5({\bf d})]([M(a_1,\cdots,
a_n)])={\mathop{\sum}\limits_{m=1,\atop
a_{m}=5}^n}[M(a_1,\cdots,a_{m-1},a_{m}-1,a_{m+1},\cdots,a_n)].
\end{eqnarray}
\begin{eqnarray}\label{E20}
[\F_3^7({\bf d})]([M(a_1,\cdots,
a_n)])={\mathop{\sum}\limits_{m=1,\atop
a_{m}=7}^n}[M(a_1,\cdots,a_{m-1},a_{m}-1,a_{m+1},\cdots,a_n)].
\end{eqnarray}

Indeed, we check the formula (\ref{E1}) as follows:
\begin{eqnarray*}
&&\quad[\E_1^{+2}({\bf d})]([M(a_1,\cdots, a_n)])\\
&&=[\textrm{proj}_{\theta_{{{\bf d}_2^4}}}(Sym^2(\Lo)\otimes
M(a_1,\cdots, a_n))]\\
&&=[\textrm{proj}_{\theta_{{{\bf d}_2^4}}}]
({\mathop{\sum}\limits_{1\leq i\leq j\leq
n}}[M(a_1,\cdots,a_{i-1},a_{i}+1,a_{i+1},
\cdots,a_{j-1},a_{j}+1,a_{j+1},\cdots,a_n)])\\
&&={\mathop{\sum}\limits_{m=1,\atop
a_{m}=2}^n}[M(a_1,\cdots,a_{m-1},a_{m}+2,a_{m+1},\cdots,a_n)]\\
&&\quad+{\mathop{\sum}\limits_{1\leq i<j\leq n,\atop
(a_{i},a_{j})=(2,3)~{\rm
or}~(3,2)}}[M(a_1,\cdots,a_{i-1},a_{i}+1,a_{i+1},\cdots,
a_{j-1},a_{j}+1,a_{j+1},\cdots,a_n)].
\end{eqnarray*}
Note that the second equality is obtained by Proposition \ref{p4},
while others are obvious by the definitions of the projective
functors $\E_1^{+2}({\bf d})$, $F_{Sym^2(\Lo)}$ and
$\textrm{proj}_{\theta}$.

Similarly, we can get the formulas from (\ref {E2}) to (\ref {E20}).

Now, by the formulas (\ref{E1})--(\ref {E20}), a categorification of
the  action of $U(\mf{so}(7,\C))$ on the $n$-th tensor power of its
spin representation can be obtained as follows.

\begin{prop}
\label{t2}
\begin{enumerate}
\item
For any $1\leq i\leq 3$ and $[{\bf d}]\in \widetilde{{\bf D}}$, the
action of $h_i$ on $(V_{\sp}^{\otm n})_{[{\bf d}]}$ can be
categorified by the exact functor $\H_i^{\sgn(c_i({\bf d}))}([{\bf
d}])$, which means that the following diagram commutes$:$
\[
\begin{array}{ccc}
K(\O_{[{\bf d}]}) &
\stackrel{}{\xrightarrow[\quad\quad\quad]{\gamma_n}}
 & ^{\Z}(V_{\sp}^{\otm n})_{[{\bf d}]}  \\
\Big\downarrow\vcenter{ \llap{$[\H_i^{\sgn(c_i({\bf d}))}({[{\bf
d}]})]~$}}     & &
\Big\downarrow\vcenter{ \rlap{$\sgn(c_i({\bf d}))h_i$}}  \\
K(\O_{[\bf d]}) &
\stackrel{}{\xrightarrow[\quad\quad\quad]{\gamma_n}} &
^{\Z}(V_{\sp}^{\otm n})_{[\bf d]}.
\end{array}
\]
\item
\begin{enumerate}
\item
For any $[{\bf d}]\in \widetilde{{\bf D}},$ the restriction of $e_1$
from $(V_{\sp}^{\otm n})_{[{\bf d}]}$ to $(V_{\sp}^{\otm n})_{\bf
[\overleftarrow{{\bf d_1}}]}$ can be categorified by a pair of exact
functors $(\E^+_{1}([{\bf d}]),\E^-_{1}([{\bf d}]))$, which means
that the following diagram commutes$:$
\[
\begin{array}{ccc}
K(\O_{[{\bf d}]}) & \stackrel{}
{\xrightarrow[\quad\quad\quad\quad\quad\quad]{\gamma_n}} & ^{\Z}(V_{\sp}^{\otm n})_{[{\bf d}]}  \\
\Big\downarrow\vcenter{ \llap{$[\E^+_{1}([{\bf d}])]-[\E^-_{1}({\bf
[d]})]~$}}     & &
\Big\downarrow\vcenter{ \rlap{$e_1$}}  \\
K(\O_{\bf [\overleftarrow{{\bf d_1}}]}) &
\stackrel{}{\xrightarrow[\quad\quad\quad\quad\quad\quad]{\gamma_n}}
& ^{\Z}(V_{\sp}^{\otm n})_{\bf [\overleftarrow{{\bf d_1}}]}.
\end{array}
\]
\item
For any $2\leq i\leq 3$ and $[{\bf d}]\in \widetilde{{\bf D}}$, the
restriction of $e_i$ from $(V_{\sp}^{\otm n})_{[{\bf d}]}$ to
$(V_{\sp}^{\otm n})_{\bf [\overleftarrow{{\bf d_i}}]}$ can be
categorified by the exact functor $\E_{i}([{\bf d}])$, which means
that the following diagram commutes$:$
\[
\begin{array}{ccc}
K(\O_{[{\bf d}]}) & \stackrel{\gamma_n}{\longrightarrow} & ^{\Z}(V_{\sp}^{\otm n})_{[{\bf d}]}  \\
\Big\downarrow\vcenter{ \llap{$[\E_{i}([{\bf d}])]~$}} & &
\Big\downarrow\vcenter{ \rlap{$e_i$}}  \\
K(\O_{\bf [\overleftarrow{{\bf d_i}}]}) &
\stackrel{\gamma_n}{\longrightarrow} &  ^{\Z}(V_{\sp}^{\otm n})_{\bf
[\overleftarrow{{\bf d_i}}]}.
\end{array}
\]
\end{enumerate}
\item
\begin{enumerate}
\item
For any $[{\bf d}]\in \widetilde{{\bf D}},$ the restriction of $f_1$
from $(V_{\sp}^{\otm n})_{[{\bf d}]}$ to $(V_{\sp}^{\otm n})_{\bf
[\overrightarrow{{\bf d_1}}]}$ can be categorified by a pair of
exact functors $(\F^+_{1}([{\bf d}]),\F^-_{1}([{\bf d}]))$, which
means that the following diagram commutes$:$
\[
\begin{array}{ccc}
K(\O_{[{\bf d}]}) & \stackrel{}
{\xrightarrow[\quad\quad\quad\quad\quad\quad]{\gamma_n}} & ^{\Z}(V_{\sp}^{\otm n})_{[{\bf d}]}  \\
\Big\downarrow\vcenter{ \llap{$[\F^+_{1}([{\bf d}])]-[\F^-_{1}({\bf
[d]})]~$}}     & &
\Big\downarrow\vcenter{ \rlap{$f_1$}}  \\
K(\O_{\bf [\overrightarrow{{\bf d_1}}]}) &
\stackrel{}{\xrightarrow[\quad\quad\quad\quad\quad\quad]{\gamma_n}}
& ^{\Z}(V_{\sp}^{\otm n})_{\bf [\overrightarrow{{\bf d_1}}]}.
\end{array}
\]
\item
For any $2\leq i\leq 3$ and $[{\bf d}]\in \widetilde{{\bf D}}$,
the restriction of $f_i$ from $(V_{\sp}^{\otm n})_{[{\bf d}]}$ to
$(V_{\sp}^{\otm
n})_{\bf
 [\overrightarrow{{\bf d_i}}]}$ can be categorified by the exact
functor $\F_{i}([{\bf d}])$, which means that the following diagram
commutes$:$
\[
\begin{array}{ccc}
K(\O_{[{\bf d}]}) & \stackrel{\gamma_n}{\longrightarrow} & ^{\Z}(V_{\sp}^{\otm n})_{[{\bf d}]}  \\
\Big\downarrow\vcenter{ \llap{$[\F_{i}([{\bf d}])]~$}} & &
\Big\downarrow\vcenter{ \rlap{$f_i$}}  \\
K(\O_{\bf [\overrightarrow{{\bf d_i}}]}) &
\stackrel{\gamma_n}{\longrightarrow} &  ^{\Z}(V_{\sp}^{\otm n})_{\bf
[\overrightarrow{{\bf d_i}}]}.
\end{array}
\]
\end{enumerate}
\end{enumerate}
\end{prop}

\begin{proof}
Here we check $(1)$ and $(2)$ in some cases.
Other cases can be verified similarly.

(1) To check $\gamma_n\circ[\H_i^{\sgn(c_i({\bf d}))}({[{\bf
d}]})]=\sgn(c_i({\bf d}))h_i\circ\gamma_n$ is equivalent to check
$$\gamma_n\circ[\H_i^{\sgn(c_i({\bf d}))}({[{\bf
d}]})]([M(a_1,\cdots, a_n)])=\sgn(c_i({\bf
d}))h_i\circ\gamma_n([M(a_1,\cdots, a_n)])$$ for any $[M(a_1,\cdots,
a_n)]\in B_{[{\bf d}]}={\mathop{\cup}\limits_{{\bf d^{'}}\in [{\bf
d}]}}B_{\bf d^{'}}.$ In fact, we have
\begin{eqnarray*}
&\quad&\gamma_n\circ[\H_i^{\sgn(c_i({\bf d}))}({[{\bf
d}]})]([M(a_1,\cdots,
a_n)])\\
&=&|c_i({\bf d})|\gamma_n([M(a_1,\cdots, a_n)])\\
&=&|c_i({\bf d})|(v_{a_1}\otimes v_{a_2}\otimes\cdots \otimes
v_{a_n})\\
&=&\sgn(c_i({\bf d}))c_i({\bf d})(v_{a_1}\otimes
v_{a_2}\otimes\cdots
\otimes v_{a_n})\\
&=&\sgn(c_i({\bf d}))h_i\circ\gamma_n([M(a_1,\cdots, a_n)]).
\end{eqnarray*}

(2) (a) To verify the commutativity of the diagram in (a), it
suffices to check $\gamma_{n}\circ([\E^+_{1}([{\bf
d}])]-[\E^-_{1}({\bf [d]})])=e_1\circ\gamma_{n}.$ Indeed, for any
$[{\bf d}]\in \widetilde{{\bf D}}$ and $[M(a_1,\cdots, a_n)]\in
B_{[{\bf d}]}={\mathop{\cup}\limits_{{\bf d^{'}}\in [{\bf
d}]}}B_{\bf d^{'}}$, we assume $[M(a_1,\cdots, a_n)]\in B_{{\bf
d}_{0}}$ for some ${\bf d}_{0}\in [{\bf d}]$, by (\ref {E1})--(\ref
{E4}), we have the following formulas:
\begin{eqnarray*}
&\quad&[\E^+_{1}([{\bf d}])]([M(a_1,\cdots, a_n)])\\
&=&[{\mathop{\oplus}\limits_{{\bf
d^{'}}\in[{\bf d}]}}(\E_1^{+2}({\bf d^{'}})\oplus\E_1^{+3}({\bf d^{'}}))M(a_1,\cdots, a_n)]\\
&=&[(\E_1^{+2}({\bf d}_{0})\oplus\E_1^{+3}({\bf d}_{0}))M(a_1,\cdots,a_n)]\\
&=&{\mathop{\sum}\limits_{m=1,\atop
a_{m}=2}^n}[M(a_1,\cdots,a_{m-1},a_{m}+2,a_{m+1},\cdots,a_n)]\\
&&+{\mathop{\sum}\limits_{m=1,\atop a_{m}=3}^n}[M(a_1,\cdots,a_{m-1},a_{m}+2,a_{m+1},\cdots,a_n)]\\
&&+{\mathop{\sum}\limits_{1\leq i<j\leq n,\atop
(a_{i},a_{j})=(2,3)~{\rm
or}~(3,2)}}[M(a_1,\cdots,a_{i-1},a_{i}+1,a_{i+1},\cdots,
a_{j-1},a_{j}+1,a_{j+1},\cdots,a_n)]\\
&&+{\mathop{\sum}\limits_{1\leq i<j\leq n,\atop
(a_{i},a_{j})=(3,4)~{\rm
or}~(4,3)}}[M(a_1,\cdots,a_{i-1},a_{i}+1,a_{i+1},\cdots,
a_{j-1},a_{j}+1,a_{j+1},\cdots,a_n)],
\end{eqnarray*}
and
\begin{eqnarray*}
&\quad&[\E^-_{1}([{\bf d}])]([M(a_1,\cdots, a_n)])\\
&=&[{\mathop{\oplus}\limits_{{\bf
d^{'}}\in[{\bf d}]}}(\E_1^{-2}({\bf d^{'}})\oplus\E_1^{-3}({\bf d^{'}}))M(a_1,\cdots, a_n)]\\
&=&[(\E_1^{-2}({\bf d}_{0})\oplus\E_1^{-3}({\bf d}_{0}))M(a_1,\cdots,a_n)]\\
&=&{\mathop{\sum}\limits_{1\leq i<j\leq n,\atop
(a_{i},a_{j})=(2,3)~{\rm
or}~(3,2)}}[M(a_1,\cdots,a_{i-1},a_{i}+1,a_{i+1},\cdots,
a_{j-1},a_{j}+1,a_{j+1},\cdots,a_n)]\\
&&+{\mathop{\sum}\limits_{1\leq i<j\leq n,\atop
(a_{i},a_{j})=(3,4)~{\rm
or}~(4,3)}}[M(a_1,\cdots,a_{i-1},a_{i}+1,a_{i+1},\cdots,
a_{j-1},a_{j}+1,a_{j+1},\cdots,a_n)].
\end{eqnarray*}

It follows that
\begin{eqnarray*}
&\quad&\gamma_{n}\circ([\E^+_{1}([{\bf d}])]-[\E^-_{1}({\bf
[d]})])([M(a_1,\cdots, a_n)])\\
&=&\gamma_{n}({\mathop{\sum}\limits_{m=1,\atop
a_{m}=2}^n}[M(a_1,\cdots,a_{m-1},a_{m}+2,a_{m+1},\cdots,a_n)])
+\gamma_{n}({\mathop{\sum}\limits_{m=1,\atop
a_{m}=3}^n}[M(a_1,\cdots,a_{m-1},a_{m}+2,a_{m+1},\cdots,a_n)])\\
&=&{\mathop{\sum}\limits_{m=1,\atop
a_{m}=2}^n}(v_{a_1}\otimes\cdots\otimes v_{a_{m-1}}\otimes
v_{a_{m}+2}\otimes v_{a_{m+1}}\otimes\cdots\otimes v_{a_{n}})
+{\mathop{\sum}\limits_{m=1,\atop
a_{m}=3}^n}(v_{a_1}\otimes\cdots\otimes v_{a_{m-1}}\otimes
v_{a_{m}+2}\otimes v_{a_{m+1}}\otimes\cdots\otimes v_{a_{n}}).
\end{eqnarray*}

Moreover, we have
\begin{eqnarray*}
&\quad&e_1\circ\gamma_{n}([M(a_1,\cdots, a_n)])\\
&=&e_1(v_{a_1}\otimes\cdots\otimes v_{a_{m-1}}\otimes
v_{a_{m}}\otimes v_{a_{m+1}}\otimes\cdots\otimes
v_{a_{n}})\\
&=&{\mathop{\sum}\limits_{m=1}^{n}}(v_{a_1}\otimes\cdots\otimes
v_{a_{m-1}}\otimes e_1v_{a_{m}}\otimes
v_{a_{m+1}}\otimes\cdots\otimes v_{a_{n}})\\
&=&{\mathop{\sum}\limits_{m=1,\atop
a_{m}=2}^n}(v_{a_1}\otimes\cdots\otimes v_{a_{m-1}}\otimes
v_{a_{m}+2}\otimes v_{a_{m+1}}\otimes\cdots\otimes v_{a_{n}})\\
&&+{\mathop{\sum}\limits_{m=1,\atop
a_{m}=3}^n}(v_{a_1}\otimes\cdots\otimes v_{a_{m-1}}\otimes
v_{a_{m}+2}\otimes v_{a_{m+1}}\otimes\cdots\otimes v_{a_{n}}).
\end{eqnarray*}
Therefore $\gamma_{n}\circ([\E^+_{1}([{\bf d}])]-[\E^-_{1}({\bf
[d]})])=e_1\circ\gamma_{n}.$

(b) We give the proof of the case $i=3.$

In the following we will prove $\gamma_{n}\circ[\E_{3}([{\bf
d}])]=e_3\circ\gamma_{n}$ which means the diagram in this case
commutes. In fact, for any $[{\bf d}]\in \widetilde{{\bf D}}$ and
$[M(a_1,\cdots, a_n)]\in B_{[{\bf d}]}={\mathop{\cup}\limits_{{\bf
d^{'}}\in [{\bf d}]}}B_{\bf d^{'}}$, we only need to check
$$\gamma_{n}\circ[\E_{3}([{\bf d}])]([M(a_1,\cdots,
a_n)])=e_3\circ\gamma_{n}([M(a_1,\cdots, a_n)]).$$ Assume that
$[M(a_1,\cdots, a_n)]\in B_{{\bf d}_{0}}$ for some ${\bf d}_{0}\in
[{\bf d}]$, by (\ref {E7})--(\ref {E10}), we have

\begin{eqnarray*}
&\quad&[\E_{3}([{\bf d}])]([M(a_1,\cdots, a_n)])\\
&=&[{\mathop{\oplus}\limits_{{\bf d^{'}}\in[{\bf d}]}}(\E_3^{0}({\bf
d^{'}})\oplus\E_3^{2}({\bf d^{'}})
\oplus\E_3^{4}({\bf d^{'}})\oplus\E_3^{6}({\bf d^{'}}))M(a_1,\cdots, a_n)]\\
&=&[(\E_3^{0}({\bf d}_{0})\oplus\E_3^{2}({\bf d}_{0})
\oplus\E_3^{4}({\bf d}_{0})\oplus\E_3^{6}({\bf d}_{0}))M(a_1,\cdots, a_n)]\\
&=&{\mathop{\sum}\limits_{m=1,\atop
a_{m}=0}^n}[M(a_1,\cdots,a_{m-1},a_{m}+1,a_{m+1},\cdots,a_n)]
+{\mathop{\sum}\limits_{m=1,\atop
a_{m}=2}^n}[M(a_1,\cdots,a_{m-1},a_{m}+1,a_{m+1},\cdots,a_n)]\\
&&+{\mathop{\sum}\limits_{m=1,\atop
a_{m}=4}^n}[M(a_1,\cdots,a_{m-1},a_{m}+1,a_{m+1},\cdots,a_n)]
+{\mathop{\sum}\limits_{m=1,\atop
a_{m}=6}^n}[M(a_1,\cdots,a_{m-1},a_{m}+1,a_{m+1},\cdots,a_n)].\\
\end{eqnarray*}

It follows that
\begin{eqnarray*}
&\quad&\gamma_{n}\circ[\E_{3}([{\bf d}])]([M(a_1,\cdots, a_n)])\\
&=&\gamma_{n}({\mathop{\sum}\limits_{m=1,\atop
a_{m}=0}^n}[M(a_1,\cdots,a_{m-1},a_{m}+1,a_{m+1},\cdots,a_n)])+\gamma_{n}({\mathop{\sum}\limits_{m=1,\atop
a_{m}=2}^n}[M(a_1,\cdots,a_{m-1},a_{m}+1,a_{m+1},\cdots,a_n)])\\
&&+\gamma_{n}({\mathop{\sum}\limits_{m=1,\atop
a_{m}=4}^n}[M(a_1,\cdots,a_{m-1},a_{m}+1,a_{m+1},\cdots,a_n)])+\gamma_{n}({\mathop{\sum}\limits_{m=1,\atop
a_{m}=6}^n}[M(a_1,\cdots,a_{m-1},a_{m}+1,a_{m+1},\cdots,a_n)])\\
&=&{\mathop{\sum}\limits_{m=1,\atop
a_{m}=0}^n}(v_{a_1}\otimes\cdots\otimes v_{a_{m-1}}\otimes
v_{a_{m}+2}\otimes v_{a_{m+1}}\otimes\cdots\otimes v_{a_{n}})
+{\mathop{\sum}\limits_{m=1,\atop
a_{m}=2}^n}(v_{a_1}\otimes\cdots\otimes v_{a_{m-1}}\otimes
v_{a_{m}+2}\otimes v_{a_{m+1}}\otimes\cdots\otimes v_{a_{n}})\\
&&+{\mathop{\sum}\limits_{m=1,\atop
a_{m}=4}^n}(v_{a_1}\otimes\cdots\otimes v_{a_{m-1}}\otimes
v_{a_{m}+2}\otimes v_{a_{m+1}}\otimes\cdots\otimes
v_{a_{n}})+{\mathop{\sum}\limits_{m=1,\atop
a_{m}=6}^n}(v_{a_1}\otimes\cdots\otimes v_{a_{m-1}}\otimes
v_{a_{m}+2}\otimes v_{a_{m+1}}\otimes\cdots\otimes v_{a_{n}}).
\end{eqnarray*}

Moreover, we have
\begin{eqnarray*}
&\quad&e_3\circ\gamma_{n}([M(a_1,\cdots, a_n)])\\
&=&e_3(v_{a_1}\otimes\cdots\otimes v_{a_{m-1}}\otimes
v_{a_{m}}\otimes v_{a_{m+1}}\otimes\cdots\otimes
v_{a_{n}})\\
&=&{\mathop{\sum}\limits_{m=1}^{n}}(v_{a_1}\otimes\cdots\otimes
v_{a_{m-1}}\otimes e_3v_{a_{m}}\otimes
v_{a_{m+1}}\otimes\cdots\otimes v_{a_{n}})\\
&=&{\mathop{\sum}\limits_{m=1,\atop
a_{m}=0}^n}(v_{a_1}\otimes\cdots\otimes v_{a_{m-1}}\otimes
v_{a_{m}+1}\otimes v_{a_{m+1}}\otimes\cdots\otimes v_{a_{n}})
+{\mathop{\sum}\limits_{m=1,\atop
a_{m}=2}^n}(v_{a_1}\otimes\cdots\otimes v_{a_{m-1}}\otimes
v_{a_{m}+1}\otimes v_{a_{m+1}}\otimes\cdots\otimes v_{a_{n}})\\
&&+{\mathop{\sum}\limits_{m=1,\atop
a_{m}=4}^n}(v_{a_1}\otimes\cdots\otimes v_{a_{m-1}}\otimes
v_{a_{m}+1}\otimes v_{a_{m+1}}\otimes\cdots\otimes v_{a_{n}})
+{\mathop{\sum}\limits_{m=1,\atop
a_{m}=6}^n}(v_{a_1}\otimes\cdots\otimes v_{a_{m-1}}\otimes
v_{a_{m}+1}\otimes v_{a_{m+1}}\otimes\cdots\otimes v_{a_{n}}).
\end{eqnarray*}
Hence $\gamma_{n}\circ[\E_{3}([{\bf d}])]([M(a_1,\cdots,
a_n)])=e_3\circ\gamma_{n}([M(a_1,\cdots, a_n)]).$
\end{proof}

\begin{thm}\label{c1}
For any $1\leq i\leq 3$ and $2\leq j\leq 3,$ let
\[\H^+_i={\mathop{\oplus}\limits_{[{\bf d}]\in
\widetilde{{\bf D}},\atop \sgn(c_i({\bf d}))=1~{\rm
or}~0}}\H_i^{\sgn(c_i({\bf d}))}([{\bf d}]):\O^{n}\to\O^{n},\quad
\H^-_i={\mathop{\oplus}\limits_{[{\bf d}]\in \widetilde{{\bf
D}},\atop \sgn(c_i({\bf d}))=-1}}\H_i^{\sgn(c_i({\bf d}))}([{\bf
d}]):\O^{n}\to\O^{n},\]
\[\E^+_{1}={\mathop{\oplus}\limits_{[{\bf d}]\in \widetilde{{\bf
D}}}}\E^+_{1}([{\bf d}]):\O^{n}\to\O^{n},\quad
\E^-_{1}={\mathop{\oplus}\limits_{[{\bf d}]\in \widetilde{{\bf
D}}}}\E^-_{1}([{\bf d}]):\O^{n}\to\O^{n},\]
\[\F^+_{1}={\mathop{\oplus}\limits_{[{\bf d}]\in
\widetilde{{\bf D}}}}\F^+_{1}([{\bf d}]):\O^{n}\to\O^{n},\quad
\F^-_{1}={\mathop{\oplus}\limits_{[{\bf d}]\in
\widetilde{{\bf D}}}}\F^-_{1}([{\bf d}]):\O^{n}\to\O^{n},\]
\[\E_{j}={\mathop{\oplus}\limits_{[{\bf d}]\in
\widetilde{{\bf D}}}}\E_{j}([{\bf d}]):\O^{n}\to\O^{n},\quad
\F_{j}={\mathop{\oplus}\limits_{[{\bf d}]\in \widetilde{{\bf
D}}}}\F_{j}([{\bf d}]):\O^{n}\to\O^{n}.\] Then we have the following
results.
\begin{enumerate}
\item
For any $1\leq i\leq 3,$ the action of $h_i$ on $V_{\sp}^{\otm n}$
can be categorified by a pair of exact functors
$(\H^+_{i},\H^-_{i}),$ which means that the following diagram commutes$:$
\[
\begin{array}{ccc}
K(\O^{n}) & \stackrel{}{\xrightarrow[\quad\quad\quad]{\gamma_n}} & ^{\Z}V_{\sp}^{\otm n} \\
\Big\downarrow\vcenter{ \llap{$[\H^+_{i}]-[\H^-_{i}]~$}} & &
\Big\downarrow\vcenter{ \rlap{$h_i$}}  \\
K(\O^{n}) & \stackrel{}{\xrightarrow[\quad\quad\quad]{\gamma_n}} &
^{\Z}V_{\sp}^{\otm n}.
\end{array}
\]
\item
\begin{enumerate}
\item
The action of $e_1$ on $V_{\sp}^{\otm n}$ can be categorified by a
pair of exact functors $(\E^+_{1},\E^-_{1})$, which means that the
following diagram commutes$:$
\[
\begin{array}{ccc}
K(\O^{n}) & \stackrel{}{\xrightarrow[\quad\quad\quad]{\gamma_n}} & ^{\Z}V_{\sp}^{\otm n} \\
\Big\downarrow\vcenter{ \llap{$[\E^+_{1}]-[\E^-_{1}]~$}} & &
\Big\downarrow\vcenter{ \rlap{$e_1$}}  \\
K(\O^{n}) & \stackrel{}{\xrightarrow[\quad\quad\quad]{\gamma_n}} &
^{\Z}V_{\sp}^{\otm n}.
\end{array}
\]
\item
For any $2\leq i\leq 3$, the action of $e_i$ on $V_{\sp}^{\otm n}$
 can be
categorified by the exact functor $\E_{i}$, which means that the
following diagram commutes$:$
\[
\begin{array}{ccc}
K(\O^{n}) & \stackrel{\gamma_n}{\longrightarrow} & ^{\Z}V_{\sp}^{\otm n} \\
\Big\downarrow\vcenter{ \llap{$[\E_{i}]~$}}     & &
\Big\downarrow\vcenter{ \rlap{$e_i$}}  \\
K(\O^{n}) & \stackrel{\gamma_n}{\longrightarrow} &
^{\Z}V_{\sp}^{\otm n}.
\end{array}
\]
\end{enumerate}
\item
\begin{enumerate}
\item
The action of $f_1$ on $V_{\sp}^{\otm n}$ can be categorified by a
pair of exact functors $(\F^+_{1},\F^-_{1})$, which means that the
following diagram commutes$:$
\[
\begin{array}{ccc}
K(\O^{n}) & \stackrel{}{\xrightarrow[\quad\quad\quad]{\gamma_n}} & ^{\Z}V_{\sp}^{\otm n} \\
\Big\downarrow\vcenter{ \llap{$[\F^+_{1}]-[\F^-_{1}]~$}} & &
\Big\downarrow\vcenter{ \rlap{$f_1$}}  \\
K(\O^{n}) & \stackrel{}{\xrightarrow[\quad\quad\quad]{\gamma_n}} &
^{\Z}V_{\sp}^{\otm n}.
\end{array}
\]
\item
For any $2\leq i\leq 3$, the action of $f_i$ on $V_{\sp}^{\otm n}$
can be categorified by the exact functor $\F_{i}$, which means that
the following diagram commutes$:$
\[
\begin{array}{ccc}
K(\O^{n}) & \stackrel{\gamma_n}{\longrightarrow} & ^{\Z}V_{\sp}^{\otm n} \\
\Big\downarrow\vcenter{ \llap{$[\F_{i}]~$}}     & &
\Big\downarrow\vcenter{ \rlap{$f_i$}}  \\
K(\O^{n}) & \stackrel{\gamma_n}{\longrightarrow} &
^{\Z}V_{\sp}^{\otm n}.
\end{array}
\]
\end{enumerate}
\end{enumerate}
\end{thm}
\begin{proof}
It is not difficult to check the diagrams are
commutative by Proposition \ref{t2}.
\end{proof}

Now we categorify defining relations of $U(\mf{so}(7,\C))$ as
natural isomorphisms between some projective functors of
$\O(\mf{gl}_n)$. Proposition \ref{p3} (5) plays an important role in the proof
of the following results.

\begin{thm}
\label{t3}
\begin{enumerate}
\item
$\H^+_i\circ\H^+_j\oplus\H^-_i\circ\H^-_j
\oplus\H^+_j\circ\H^-_i\oplus\H^-_j\circ\H^+_i
\cong
\H^+_i\circ\H^-_j\oplus\H^-_i\circ\H^+_j
\oplus\H^+_j\circ\H^+_i\oplus\H^-_j\circ\H^-_i$ for $1\leq i,j\leq 3.$
\item
\begin{enumerate}
\item
$\E^+_1\circ\F^+_1\oplus\E^-_1\circ\F^-_1
\oplus\F^+_1\circ\E^-_1\oplus\F^-_1\circ\E^+_1\oplus\H^-_1
\cong\E^+_1\circ\F^-_1\oplus\E^-_1\circ\F^+_1
\oplus\F^+_1\circ\E^+_1\oplus\F^-_1\circ\E^-_1\oplus
\H^+_1;$
\item
$\E^+_1\circ\F_j\oplus\F_j\circ\E^-_1
\cong\E^-_1\circ\F_j\oplus\F_j\circ\E^+_1$ for $j=2,3;$
\item
$\E_i\circ\F^+_1\oplus\F^-_1\circ\E_i
\cong\E_i\circ\F^-_1\oplus\F^+_1\circ\E_i$ for $i=2,3;$
\item
$\E_i\circ\F_j\oplus\delta_{i,j}\H^-_i
\cong\F_j\circ\E_i\oplus\delta_{i,j}\H^+_i$
for $(i,j)=(2,2), (2,3), (3,2)$ or $(3,3).$
\end{enumerate}

\item
\begin{enumerate}
\item
$\H^+_1\circ\E^+_1\oplus\H^-_1\circ\E^-_1
\oplus\E^+_1\circ\H^-_1\oplus\E^-_1\circ\H^+_1\oplus(\E^-_1)^{\oplus 2}
\cong
\H^+_1\circ\E^-_1\oplus\H^-_1\circ\E^+_1
\oplus\E^+_1\circ\H^+_1\oplus\E^-_1\circ\H^-_1\oplus(\E^+_1)^{\oplus
2};$
\item
$\H^+_i\circ\E^+_1\oplus\H^-_i\circ\E^-_1\oplus\E^+_1\circ\H^-_i\oplus\E^-_1\circ\H^+_i\oplus(\E^+_1)^{\oplus
(-a_{i,1})}
\cong
\H^+_i\circ\E^-_1\oplus\H^-_i\circ\E^+_1\oplus\E^+_1\circ\H^+_i\oplus\E^-_1\circ\H^-_i\oplus(\E^-_1)^{\oplus
(-a_{i,1})}$ for $i=2, 3;$
\item
$\H^+_i\circ\E_j\oplus
\E_j\circ\H^-_i\oplus(\E_j)^{\oplus(-a_{i,j})} \cong
\H^-_i\circ\E_j\oplus \E_j\circ\H^+_i$
for $(i,j)=(1,2), (1,3), (2,3)$ or $(3,2);$
\item
$\H^+_i\circ\E_i\oplus
\E_i\circ\H^-_i\cong \H^-_i\circ\E_i\oplus
\E_i\circ\H^+_i\oplus\E_i^{\oplus 2}$ for $i=2, 3.$
\end{enumerate}
\item
\begin{enumerate}
\item
$\H^+_1\circ\F^+_1\oplus\H^-_1\circ\F^-_1\oplus\F^+_1\circ\H^-_1\oplus\F^-_1
\circ\H^+_1\oplus(\F^+_1)^{\oplus 2}
\cong
\H^+_1\circ\F^-_1\oplus\H^-_1\circ\F^+_1\oplus\F^+_1\circ\H^+_1\oplus\F^-_1
\circ\H^-_1\oplus(\F^-_1)^{\oplus 2};$
\item
$\H^+_i\circ\F^+_1\oplus\H^-_i\circ\F^-_1\oplus\F^+_1\circ\H^-_i\oplus\F^-_1\circ\H^+_i\oplus(\F^-_1)^{\oplus
(-a_{i,1})}
\cong
\H^+_i\circ\F^-_1\oplus\H^-_i\circ\F^+_1\oplus\F^+_1\circ\H^+_i\oplus\F^-_1\circ\H^-_i\oplus(\F^+_1)^{\oplus
(-a_{i,1})}$ for $i=2,3;$
\item
$\H^+_i\circ\F_j\oplus
\F_j\circ\H^-_i \cong \H^-_i\circ\F_j\oplus
\F_j\circ\H^+_i\oplus(\F_j)^{\oplus(-a_{i,j})}$
 for $(i,j)=(1,2), (1,3), (2,3)$ or $(3,2);$
\item
$\H^+_i\circ\F_i\oplus
\F_i\circ\H^-_i\oplus\F_i^{\oplus 2}\cong \H^-_i\circ\F_i\oplus
\F_i\circ\H^+_i$ for $i=2, 3.$
\end{enumerate}

\item
\begin{enumerate}
\item
$\E^+_{1}\circ\E^+_{1}\circ\E_{2}\oplus\E^-_{1}\circ\E^-_{1}\circ\E_{2}\oplus
\E_{2}\circ\E^+_{1}\circ\E^+_{1}\oplus
\E_{2}\circ\E^-_{1}\circ\E^-_{1}\oplus(\E^+_{1}\circ\E_{2}\circ\E^-_{1})^{\oplus
2}\oplus(\E^-_{1}\circ\E_{2}\circ\E^+_{1})^{\oplus 2}\\
\cong\E^+_{1}\circ\E^-_{1}\circ\E_{2}\oplus\E^-_{1}\circ\E^+_{1}\circ\E_{2}
\oplus \E_{2}\circ\E^+_{1}\circ\E^-_{1}\oplus
\E_{2}\circ\E^-_{1}\circ\E^+_{1}\oplus(\E^+_{1}\circ\E_{2}\circ\E^+_{1})^{\oplus
2}\oplus(\E^-_{1}\circ\E_{2}\circ\E^-_{1})^{\oplus 2};$
\item
$\E^+_{1}\circ\E_{3}\oplus\E_{3}\circ\E^-_{1}\cong\E^-_{1}\circ\E_{3}\oplus\E_{3}\circ\E^+_{1};$
\item
$\E_{2}\circ\E_{2}\circ\E^+_{1}\oplus
\E^+_{1}\circ\E_{2}\circ\E_{2}\oplus(\E_{2}\circ\E^-_{1}\circ\E_{2})^{\oplus
2} \cong\E_{2}\circ\E_{2}\circ\E^-_{1}\oplus
\E^-_{1}\circ\E_{2}\circ\E_{2}\oplus(\E_{2}\circ\E^+_{1}\circ\E_{2})^{\oplus
2};$
\item
$\E_{2}\circ\E_{2}\circ\E_{3}\oplus
\E_{3}\circ\E_{2}\circ\E_{2}\cong(\E_{2}\circ\E_{3}\circ\E_{2})^{\oplus
2};$
\item
$\E_{3}\circ\E_{3}\circ\E_{3}\circ\E_{2}\oplus(\E_{3}\circ\E_{2}\circ\E_{3}\circ\E_{3})^{\oplus
3} \cong(\E_{3}\circ\E_{3}\circ\E_{2}\circ\E_{3})^{\oplus
3}\oplus\E_{2}\circ\E_{3}\circ\E_{3}\circ\E_{3}$.
\end{enumerate}

\item
\begin{enumerate}
\item
$\F^+_{1}\circ\F^+_{1}\circ\F_{2}\oplus\F^-_{1}\circ\F^-_{1}\circ\F_{2}\oplus
\F_{2}\circ\F^+_{1}\circ\F^+_{1}\oplus
\F_{2}\circ\F^-_{1}\circ\F^-_{1}\oplus(\F^+_{1}\circ\F_{2}\circ\F^-_{1})^{\oplus
2}\oplus(\F^-_{1}\circ\F_{2}\circ\F^+_{1})^{\oplus 2}\\
\cong\F^+_{1}\circ\F^-_{1}\circ\F_{2}\oplus\F^-_{1}\circ\F^+_{1}\circ\F_{2}
\oplus \F_{2}\circ\F^+_{1}\circ\F^-_{1}\oplus
\F_{2}\circ\F^-_{1}\circ\F^+_{1}\oplus(\F^+_{1}\circ\F_{2}\circ\F^+_{1})^{\oplus
2}\oplus(\F^-_{1}\circ\F_{2}\circ\F^-_{1})^{\oplus 2};$
\item
$\F^+_{1}\circ\F_{3}\oplus\F_{3}\circ\F^-_{1}\cong\F^-_{1}\circ\F_{3}\oplus\F_{3}\circ\F^+_{1};$
\item
$\F_{2}\circ\F_{2}\circ\F^+_{1}\oplus
\F^+_{1}\circ\F_{2}\circ\F_{2}\oplus(\F_{2}\circ\F^-_{1}\circ\F_{2})^{\oplus
2} \cong\F_{2}\circ\F_{2}\circ\F^-_{1}\oplus
\F^-_{1}\circ\F_{2}\circ\F_{2}\oplus(\F_{2}\circ\F^+_{1}\circ\F_{2})^{\oplus
2};$
\item
$\F_{2}\circ\F_{2}\circ\F_{3}\oplus
\F_{3}\circ\F_{2}\circ\F_{2}\cong(\F_{2}\circ\F_{3}\circ\F_{2})^{\oplus
2};$
\item
$\F_{3}\circ\F_{3}\circ\F_{3}\circ\F_{2}\oplus(\F_{3}\circ\F_{2}\circ\F_{3}\circ\F_{3})^{\oplus
3} \cong(\F_{3}\circ\F_{3}\circ\F_{2}\circ\F_{3})^{\oplus
3}\oplus\F_{2}\circ\F_{3}\circ\F_{3}\circ\F_{3}.$
\end{enumerate}
\end{enumerate}
\end{thm}

\begin{proof}
We check Theorem \ref{t3} in some cases, while the remaining cases
can be verified similarly.

(1) By Proposition \ref{p3} (5), we only need to check that
\begin{eqnarray}\label{equ3.2}
&\quad&[\H^+_i]\circ[\H^+_j]+[\H^-_i]\circ[\H^-_j]+[\H^+_j]\circ[\H^-_i]+[\H^-_j]\circ[\H^+_i]\\
&=&[\H^+_i]\circ[\H^-_j]+[\H^-_i]\circ[\H^+_j]+[\H^+_j]\circ[\H^+_i]+[\H^-_j]
\circ[\H^-_i].\nonumber
\end{eqnarray}
Indeed,  it is easy to see that $h_ih_j\gamma_n=h_jh_i\gamma_n.$ By
Theorem \ref{c1} (1), we have
$$h_ih_j\gamma_n=\gamma_n\circ([\H^+_i]-[\H^-_i])\circ([\H^+_j]-[\H^-_j]),$$
$$h_jh_i\gamma_n=\gamma_n\circ([\H^+_j]-[\H^-_j])\circ([\H^+_i]-[\H^-_i]).$$
Therefore,
\begin{eqnarray}\label{equ3.3}
([\H^+_i]-[\H^-_i])\circ([\H^+_j]-[\H^-_j])=([\H^+_j]-[\H^-_j])
\circ([\H^+_i]-[\H^-_i]).
\end{eqnarray}
We immediately get (\ref{equ3.2}) by expanding (\ref{equ3.3}).

(2) (a) By Proposition \ref{p3} (5), it is enough to check that
\begin{eqnarray}\label{equ3.4}
&&[\E^+_1]\circ[\F^+_1]+[\E^-_1]\circ[\F^-_1]+[\F^+_1]\circ[\E^-_1]+[\F^-_1]\circ[\E^+_1]+[\H^-_1]\\
&=&[\E^+_1]\circ[\F^-_1]+[\E^-_1]\circ[\F^+_1]+[\F^+_1]\circ[\E^+_1]+[\F^-_1]\circ[\E^-_1]+[\H^+_1].\nonumber
\end{eqnarray}
Indeed, noting that\\
(i)\quad $(e_1f_1-f_1e_1)\gamma_n=h_1\gamma_n,$ \\
(ii)\quad
$(e_1f_1-f_1e_1)\gamma_n=\gamma_n\circ\big[([\E^+_1]-[\E^-_1])\circ([\F^+_1]-[\F^-_1])
-([\F^+_1]-[\F^-_1])\circ([\E^+_1]-[\E^-_1])\big],$ \\
(iii)\quad $h_1\gamma_n=\gamma_n\circ([\H^+_1]-[\H^-_1]),$\\
where (ii) and (iii) follow from Theorem \ref{c1} (1), (2)(a) and
(3)(a), we have
\begin{eqnarray}\label{equ3.5}
([\E^+_1]-[\E^-_1])\circ([\F^+_1]-[\F^-_1])-([\F^+_1]-[\F^-_1])\circ([\E^+_1]-[\E^-_1])=[\H^+_1]-[\H^-_1],
\end{eqnarray}
(\ref{equ3.4}) is obtained by expanding (\ref{equ3.5}).

(3) (a) By Proposition \ref{p3} (5), it suffices to check that
\begin{eqnarray}\label{equ3.6}
&\quad&[\H^+_1]\circ[\E^+_1]+[\H^-_1]\circ[\E^-_1]+[\E^+_1]\circ[\H^-_1]+[\E^-_1]\circ[\H^+_1]+2[\E^-_1]\\
&=&
[\H^+_1]\circ[\E^-_1]+[\H^-_1]\circ[\E^+_1]+[\E^+_1]\circ[\H^+_1]+[\E^-_1]\circ[\H^-_1]+2[\E^+_1].\nonumber
\end{eqnarray}
Indeed,  noting that\\
(i)\quad $(h_1e_1-e_1h_1)\gamma_n=2e_1\gamma_n,$\\
(ii)\quad
$(h_1e_1-e_1h_1)\gamma_n=\gamma_n\circ\big[([\H^+_1]-[\H^-_1])\circ([\E^+_1]-[\E^-_1])
-([\E^+_1]-[\E^-_1])\circ([\H^+_1]-[\H^-_1])\big],$\\
(iii)\quad $2e_1\gamma_n=2\gamma_n\circ([\E^+_1]-[\E^-_1]),$\\
where (ii) and (iii) can be seen from Theorem \ref{c1} (1) and
(2)(a), we have
\begin{eqnarray}\label{equ3.7}
([\H^+_1]-[\H^-_1])\circ([\E^+_1]-[\E^-_1])
-([\E^+_1]-[\E^-_1])\circ([\H^+_1]-[\H^-_1])=2([\E^+_1]-[\E^-_1]),
\end{eqnarray}
(\ref{equ3.6}) follows from expanding (\ref{equ3.7}).

(4) (b) We only give the proof of the case $(i,j)=(2,1).$ By
Proposition \ref{p3} (5), it is sufficient to check that
\begin{eqnarray}\label{equ3.8}
&\quad&[\H^+_2]\circ[\F^+_1]+[\H^-_2]\circ[\F^-_1]+[\F^+_1]\circ[\H^-_2]+[\F^-_1]\circ[\H^+_2]+[\F^-_1]\\
&=&
[\H^+_2]\circ[\F^-_1]+[\H^-_2]\circ[\F^+_1]+[\F^+_1]\circ[\H^+_2]+[\F^-_1]\circ[\H^-_2]+[\F^+_1].\nonumber
\end{eqnarray}
Indeed, noting that\\
(i)\quad $(h_2f_1-f_1h_2)\gamma_n=f_1\gamma_n,$\\
(ii)\quad
$(h_2f_1-f_1h_2)\gamma_n=\gamma_n\circ\big[([\H^+_2]-[\H^-_2])\circ([\F^+_1]-[\F^-_1])
-([\F^+_1]-[\F^-_1])\circ([\H^+_2]-[\H^-_2])\big],$\\
(iii)\quad $f_1\gamma_n=\gamma_n\circ([\F^+_1]-[\F^-_1]),$\\
where (ii) and (iii) follow from Theorem \ref{c1} (1) and (3)(a), we
have
\begin{eqnarray}\label{equ3.9}
([\H^+_2]-[\H^-_2])\circ([\F^+_1]-[\F^-_1])
-([\F^+_1]-[\F^-_1])\circ([\H^+_2]-[\H^-_2])=[\F^+_1]-[\F^-_1],
\end{eqnarray}
We obtain (\ref{equ3.8}) by expanding (\ref{equ3.9}).

(5) (a) By Proposition \ref{p3} (5), it is equivalent to check
\begin{eqnarray}\label{equ3.10}
&\quad&[\E^+_{1}]\circ[\E^+_{1}]\circ[\E_{2}]+[\E^-_{1}]\circ[\E^-_{1}]\circ[\E_{2}]+
[\E_{2}]\circ[\E^+_{1}]\circ[\E^+_{1}]\\
&\quad&+[\E_{2}]\circ[\E^-_{1}]\circ[\E^-_{1}]+2[\E^+_{1}]\circ[\E_{2}]\circ[\E^-_{1}]
+2[\E^-_{1}]\circ[\E_{2}]\circ[\E^+_{1}]\nonumber\\
&=&[\E^+_{1}]\circ[\E^-_{1}]\circ[\E_{2}]+[\E^-_{1}]\circ[\E^+_{1}]\circ[\E_{2}]
+[\E_{2}]\circ[\E^+_{1}]\circ[\E^-_{1}]\nonumber\\
&\quad&+[\E_{2}]\circ[\E^-_{1}]\circ[\E^+_{1}]+2[\E^+_{1}]\circ[\E_{2}]\circ[\E^+_{1}]
+2[(\E^-_{1}]\circ[\E_{2}]\circ[\E^-_{1}]).\nonumber
\end{eqnarray}
Indeed, noting that\\
(i)\quad $(e^2_1e_2-2e_1e_2e_1+e_2e^2_1)\gamma_n=0,$\\
(ii)\quad by Theorem \ref{c1} (2) (a) and (b),
$(e^2_1e_2-2e_1e_2e_1+e_2e^2_1)\gamma_n
=\gamma_n\circ\big[([\E^+_1]-[\E^-_1])
\circ([\E^+_1]-[\E^-_1])\circ[\E_2]
-2([\E^+_1]-[\E^-_1])\circ[\E_2]\circ([\E^+_1]-[\E^-_1])
+[\E_2]\circ([\E^+_1]-[\E^-_1])\circ([\E^+_1]-[\E^-_1])\big],$\\
we have $(3.31)$
\begin{eqnarray*}
([\E^+_1]-[\E^-_1])\circ([\E^+_1]-[\E^-_1])\circ[\E_2]
-2([\E^+_1]-[\E^-_1])\circ[\E_2]\circ([\E^+_1]-[\E^-_1])
+[\E_2]\circ([\E^+_1]-[\E^-_1])\circ([\E^+_1]-[\E^-_1])=0,
\end{eqnarray*}
which is (\ref{equ3.10}) after expanding.

(6) (e) By Proposition \ref{p3} (5), it is equivalent to verify
\[[\F_{3}]\circ[\F_{3}]\circ[\F_{3}]\circ[\F_{2}]+3[\F_{3}]\circ[\F_{2}]\circ[\F_{3}]\circ[\F_{3}]
=3[\F_{3}]\circ[\F_{3}]\circ[\F_{2}]\circ[\F_{3}]+[\F_{2}]\circ[\F_{3}]
\circ[\F_{3}]\circ[\F_{3}].\]
 By Theorem \ref{c1} (3) (b), we have
\begin{eqnarray*}
&&\gamma_n\circ([\F_{3}]\circ[\F_{3}]\circ[\F_{3}]\circ[\F_{2}]
+3[\F_{3}]\circ[\F_{2}]\circ[\F_{3}]\circ[\F_{3}])\\
&&=(f^3_{3}f_{2}+3f_{3}f_{2}f^2_{3})\circ\gamma_n=(3f^2_{3}f_{2}f_{3}+f_{2}f^3_{3})\circ\gamma_n\\
&&=\gamma_n\circ(3[\F_{3}]\circ[\F_{3}]\circ[\F_{2}]\circ[\F_{3}]
+[\F_{2}]\circ[\F_{3}]\circ[\F_{3}]\circ[\F_{3}]).
\end{eqnarray*}
Thus
$[\F_{3}]\circ[\F_{3}]\circ[\F_{3}]\circ[\F_{2}]+3[\F_{3}]\circ[\F_{2}]\circ[\F_{3}]\circ[\F_{3}]
=3[\F_{3}]\circ[\F_{3}]\circ[\F_{2}]\circ[\F_{3}]+[\F_{2}]\circ[\F_{3}]
\circ[\F_{3}]\circ[\F_{3}]$ since $\gamma_n$ is an abelian group
isomorphism.

\end{proof}

\section*{Acknowledgement}

We are grateful to the referee for many useful comments, suggestions
and information about relevant references.


\begin{thebibliography}{111}

\bibitem{BFK}\label{BFK}
J. Bernstein, I. Frenkel and M. Khovanov, A categorification of the
Temperley-Lieb algebra and Schur quotients of $U(\mf sl_2)$ via
projective and Zuckerman functors, Selecta Math. (N.S.) 5 (2)
(1999), 199-241.

\bibitem{BG}\label{BG}
J. N. Bernstein and S. I. Gelfand, Tensor products of finite and
infinite dimensional representations of semisimple Lie algebras,
Compositio Math. (1980), 245-285.

\bibitem{BGG}\label{BGG}
J. N. Bernstein, I. M. Gelfand and S. I. Gelfand, Category of
$\mf{g}$-modules, Functional Anal. and Appl. 2 (1976), 87-92.

\bibitem{C}\label{C}
L. Crane, Clock and category: is quantum gravity algebraic?, Jour.
Math. Phys. 36 (1995), 6180-6193.

\bibitem{CF}\label{CF}
L. Crane and I. Frenkel, Four-dimensional topological quantum field
theory, Hopf categories, and the canonical bases, in: Topology and
Physics, J. Math. Phys. 35 (10) (1994), 5136-5154.

\bibitem{Di}\label{Di}
J. Dixmier, Enveloping algebras, Revised reprint of the 1977
translation. Graduate Studies in Mathematics, 11. American
Mathematical Society, Providence, RI, 1996.

\bibitem{FH}\label{FH}
W. Fulton and J. Harris, Representation Theory: A First Course,
Graduate Texts in Mathematics vol. 129, Springer-Verlag, 1991.


\bibitem{HK}\label{HK}
J. Hong and S. J. Kang, Introduction to Quantum Groups and Crystal
Bases, Grad. Stud. Math., vol. 42, Amer. Math. Soc., Providence, RI,
2002.

\bibitem{H}\label{H}
J. E. Humphreys, Representations of semisimple Lie algebras in the
BGG category $\O$, Grad. Stud. Math., vol. 94, Amer. Math. Soc.,
Providence, RI, 2008.

\bibitem{KMS479}\label{KMS479}
M. Khovanov, V. Mazorchuk and C. Stroppel, A brief review of abelian
categorifications, Theory Appl. Categ. 22 (2009), No. 19, 479-508.

\bibitem{KMS1163}\label{KMS1163}
M. Khovanov, V. Mazorchuk and C. Stroppel, A categorification of
integral Specht modules, Proc. Amer. Math. Soc. 136 (2008), No. 4,
1163-1169.

\bibitem{LM}\label{LM}
B. Lawson and M. Michelson, Spin Geometry, Princeton University
Press, 1989.

\bibitem{M}\label{M}
V. Mazorchuk, Lectures on algebraic categorification,
arXiv:1011.0144

\bibitem{MM}\label{MM}
V. Mazorchuk and V. Miemietz, Cell 2-representations of finitary
2-categories, Compositio Math.147 (2011), 1519-1545.

\bibitem{MS1363}\label{MS1363}
V. Mazorchuk and C. Stroppel, Categorification of (induced) cell
modules and the rough structure of generalised Verma modules, Adv.
Math. 219 (2008), No. 4, 1363-1426.

\bibitem{MS1}\label{MS1}
V. Mazorchuk and C. Stroppel, Categorification of Wedderburn's basis
for $\C[S_n]$, Arch. Math. (Basel) 91 (2008), No. 1, 1-11.

\bibitem{MS2939}\label{MS2939}
V. Mazorchuk and C. Stroppel, Translation and shuffling of
projectively presentable modules and a categorification of a
parabolic Hecke module, Trans. Amer. Math. Soc. 357 (2005), No. 7,
2939-2973.

\bibitem{R}\label{R} R. Rouquier, 2-Kac-Moody algebras, arXiv:math.RT/0812.5023.

\bibitem{S}\label{S}
C. Stroppel, Categorification of the Temperley-Lieb category,
tangles, and cobordisms via projective functors, Duke Math. J. 126
(3) (2005), 547-596.

\bibitem{SU}\label{SU}
J. Sussan, Category $\O$ and $\mf sl_k$-link invariants,
arXiv:math.QA/0701045.

\end{thebibliography}
\end{document}